\documentclass[11pt]{amsart}

\usepackage{a4wide}
\usepackage{amssymb}
\usepackage{mathrsfs}
\usepackage{enumerate, mathtools}
\usepackage{esint}
\usepackage{titletoc}
\usepackage[colorlinks=true, urlcolor=blue, linkcolor=blue, citecolor=red]{hyperref}
\usepackage[nameinlink]{cleveref}

\theoremstyle{plain}
\newtheorem{theorem}{Theorem}[section]
\newtheorem{lemma}[theorem]{Lemma}

\theoremstyle{definition}
\newtheorem{definition}[theorem]{Definition}
\newtheorem{remark}[theorem]{Remark}

\numberwithin{equation}{section}

\usepackage{nameref}
\makeatletter
\let\orgdescriptionlabel\descriptionlabel
\renewcommand*{\descriptionlabel}[1]{%
	\let\orglabel\label
	\let\label\@gobble
	\phantomsection
	\edef\@currentlabel{#1}%
	\let\label\orglabel
	\orgdescriptionlabel{#1}%
}
\makeatother
\numberwithin{equation}{section}

\DeclareMathOperator{\data}{\textbf{data}}

\DeclareMathOperator{\diam}{diam}
\DeclareMathOperator{\tr}{tr}

\def\O{\Omega}

\newcommand{\R}{\mathbb{R}}
\newcommand{\norm}[1]{\left\lVert#1\right\rVert}
\newcommand{\inner}[2]{\left\langle #1, #2 \right\rangle}

\newcommand{\I}{\int\limits}

\begin{document}
	
\title[]{Global regularity results for a class of singular/degenerate fully nonlinear elliptic equations}

\everymath{\displaystyle}

\author{Sumiya Baasandorj}
\address{Department of Mathematical Sciences, Seoul National University, Seoul 08826, Korea}
\email{summa2017@snu.ac.kr}

\author{Sun-Sig Byun}
\address{Department of Mathematical Sciences and Research Institute of Mathematics,
Seoul National University, Seoul 08826, Korea}
\email{byun@snu.ac.kr}

\author{Ki-Ahm Lee}
\address{Department of Mathematical Sciences and Research Institute of Mathematics,
	Seoul National University, Seoul 08826, Korea.}
\email{kiahm@snu.ac.kr}

\author{Se-Chan Lee}
\address{Research Institute of Mathematics,
	Seoul National University, Seoul 08826, Korea.}
\email{dltpcks1@snu.ac.kr}


\begin{abstract}
	We provide the Alexandroff-Bakelman-Pucci estimate and global $C^{1, \alpha}$-regularity for a class of singular/degenerate fully nonlinear elliptic equations. We also derive the existence of a viscosity solution to the Dirichlet problem with the associated operator.
\end{abstract}

\maketitle
\tableofcontents

\section{Introduction}\label{intro}
In this paper, we are concerned with the existence and global regularity results for viscosity solutions of a class of singular/degenerate fully nonlinear equations of the form
\begin{align}
	\label{me}
		\left\{\begin{array}{rclcc}
		\Phi(x,|Du|)F(D^2u) &=& f(x) & \text{in }& \O, \\
		u(x)&=&g(x) & \text{on }& \partial\O,
		\end{array}\right.	  
	\end{align}
where $F : \mathcal{S}(n)\rightarrow \R$ is a uniformly $(\lambda,\Lambda)$-elliptic operator in the sense of \ref{a1}, $\Phi : \Omega\times [0,\infty)\rightarrow [0,\infty)$ is a continuous map featuring degeneracy and singularity for the gradient described as in \ref{a2}, $f(\cdot)$ and $g(\cdot)$ are suitable regular functions in the sense of \ref{a3}, and $\Omega$ is a $C^{1,1}$-domain as in \ref{a4}. We recall that, as a consequence of Krylov-Safonov theory \cite{CC1}, viscosity solutions to the homogeneous equation
\begin{align*}
F(D^{2}u)=0 \quad \text{in $B_{1}$}, \quad \text{where $F$ is uniformly $( \lambda, \Lambda)$-elliptic,}
\end{align*}  
belong to $C^{1, \overline{\alpha}}_{\mathrm{loc}}(B_{1})$ for a universal constant $\overline{\alpha} \equiv \overline{\alpha}(n, \lambda, \Lambda) \in (0,1)$.

Some special cases of \eqref{me}, which are  singular or degenerate PDEs in non-divergence structure, have been widely studied in the past years. To be precise, the local $C^{1, \alpha}$-regularity results for degenerate fully nonlinear equations were developed in \cite{ART1, IS1} for $\Phi(x,t)=t^{p}$ with $p \geq 0$, in \cite{SR1, De1} for $\Phi(x,t)=t^{p}+a(x)t^{q}$ with $0 \leq p \leq q$,  in \cite{BPRT1} for $\Phi(x,t)=t^{p(x)}$ with $\inf p(\cdot)>-1$, and in \cite{dSJRR1, FRZ1} for $\Phi(x,t)=t^{p(x)}+a(x)t^{q(x)}$ with $0 \leq p(\cdot) \leq q(\cdot)$. On the other hand, comparison principle, Liouville type results, and the ABP estimate are found mostly for $\Phi(x,t)=t^{p}$ with $-1<p<0$; we refer to \cite{BD1, BD2, DFQ1, DFQ2, Im1} for details.  Finally, for both singular and degenerate general operators which are considered in this paper, the local $C^{1, \alpha}$-regularity with the optimality was shown by the authors \cite{BBLL1}. Global counterpart of such local regularity results can be found in \cite{BD4} for $\Phi(x,t)=t^{p}$ with $p \geq 0$ and in \cite{dSJRR1} for $\Phi(x,t)=t^{p(x)}+a(x)t^{q(x)}$ with $0 \leq p(\cdot) \leq q(\cdot)$. We remark that global regularity of viscosity solutions to \eqref{me} plays an essential role in proving the solvability of the Dirichlet problem and in investigating obstacle problems; see \cite{SV2, SV1, dSJRR1}.

The goal of this paper is to investigate the global regularity, involving the ABP estimate and $C^{1, \alpha}$-estimate up to the boundary, for both singular and degenerate fully nonlinear elliptic equations in a unified way. Our first main theorem in this paper reads as follows:
\begin{theorem}[Global $C^{1, \alpha}$-regularity] \label{thm:main}
	Suppose the assumptions \ref{a1}-\ref{a4} (to be described in the next section) are in force. Let $\alpha$ be chosen to satisfy
	\begin{align} \label{alpha}
	\alpha \in
	\begin{cases}
	 (0, \overline{\alpha}) \cap \left(0, \frac{1}{1+s(\Phi)} \right] \cap (0, \beta_g) \ &\text{if $i(\Phi) \geq 0$},\\
	(0, \overline{\alpha}) \cap \left(0, \frac{1}{1+s(\Phi)-i(\Phi)} \right] \cap (0, \beta_g) \ &\text{if $-1<i(\Phi) <0$}.
	\end{cases}
	\end{align} 
	For any viscosity solution $u$ of
	\begin{align*}
	\begin{cases}
	\Phi(x, |Du|)F(D^2u)=f \ &\text{in $\Omega$},\\
	u=g \ &\text{on $\partial \Omega$},
	\end{cases}
	\end{align*}
	there exists a constant $c \equiv c(n, \lambda, \Lambda, i(\Phi), L, \alpha)$ such that $u \in C^{1,\alpha}(\overline{\Omega})$ and 
	\begin{align*}
	\|u\|_{C^{1,\alpha}(\overline{\Omega})} \leq c\left(1+\norm{u}_{L^{\infty}(\Omega)} + \|g\|_{C^{1, \beta_{g}}(\partial \Omega)} + \norm{f/\nu_{0}}_{L^{\infty}(\Omega)}^{\frac{1}{1+i(\Phi)}}\right).
	\end{align*}
\end{theorem}

The second main theorem concerning the solvability of the Dirichlet problem follows from \Cref{thm:main} together with Perron's method.
\begin{theorem}[Existence of viscosity solution]
	\label{thm:mthm1}
	Suppose the assumptions \ref{a1}-\ref{a4} and \ref{a5} (to be stated in \Cref{Perron}) are in force. Then there exists a viscosity solution $u\in C(\overline{\O})$ of \eqref{me}.
\end{theorem}

Our strategy is to improve the global regularity of a viscosity solution $u$ gradually. For this purpose, we begin with the ABP estimate to show the global boundedness of solutions. Then, by constructing an appropriate barrier function near the boundary, we capture the boundary behavior of solutions in terms of a distance function. The comparison with a distance function allows us to achieve a global Lipschitz estimate. In the end, we prove the approximation lemma by employing the compactness argument and then determine approximating linear functions in an iterative manner. 

The main difficulty arises due to different behaviors of solutions relying on the sign of $i(\Phi)$ defined in \ref{a2}.
To overcome such a challenge, we first discuss the degenerate case ($i(\Phi) \geq 0$) in \Cref{C01} and then transport the regular properties to a viscosity solution of the singular case ($-1<i(\Phi)<0$) along with a suitable modification of equations in \Cref{C1alpha}. In addition, the degenerate or singular character of PDEs leads to the lack of the comparison principle in general settings. Therefore, we formulate special types of the comparison principle: one is \Cref{specialCP}, where we exploit the smooth feature of barrier functions, and the other is \Cref{thm_CP}, in which we approximate the equations to have a monotone property with respect to viscosity solution $u$ of \eqref{me},

The paper is organized as follows. In \Cref{preliminaries}, we present the assumptions \ref{a1}-\ref{a4} on the equation \eqref{me} and data to be used throughout the paper, and then collect preliminary results related to our main theorem. \Cref{ABP} is devoted to the proof of ABP estimate. The proofs for global $C^{0, 1}$-estimate and $C^{1, \alpha}$-estimate of viscosity solutions $u$ of \eqref{me} are provided in \Cref{C01} and \Cref{C1alpha}, respectively. Finally, in \Cref{Perron}, we prove the comparison principle under an additional assumption \ref{a5} to deduce the existence of a viscosity solution by Perron's method.  

\section{Preliminaries}
\label{preliminaries}
Throughout the paper, we denote by $B_{r}(x_0):= \{x\in \R^{n} : |x-x_0|< r\}$ the open ball of $\R^{n}$ with $n\geq 2$ centered at $x_0$ with positive radius $r$. If the center is clear in the context, we shall omit the center point by writing $B_{r}\equiv B_{r}(x_0)$. Moreover, $B_{1}\equiv B_{1}(0)\subset \R^{n}$ denote the unit ball. We shall always denote by $c$ a generic positive constant, possible varying line to line, having dependencies on parameters using brackets, that is, for example $c\equiv c(n,i(\Phi),\nu_0)$ means that $c$ depends only on $n,i(\Phi)$ and $\nu_0$. For two positive functions $f, g$, we write $f \lesssim g$ when there exists a universal constant $c>0$ such that $f \leq cg$.

For a measurable map $g : \mathcal{B}\subset B_{1}\rightarrow \R^{N}$ $(N\geq 1)$ with $\beta\in (0,1]$ being a given number, we shall use the notation
\begin{align*}
	[g]_{C^{0,\beta}(\mathcal{B})}:= \sup\limits_{x\neq y\in \mathcal{B}} \frac{|g(x)-g(y)|}{|x-y|^{\beta}},\quad
	[g]_{C^{0,\beta}}:= [g]_{C^{0,\beta}(B_{1})}.
\end{align*}

We now state the main assumptions in the paper.
\begin{description}
	\item[(A1)\label{a1}] The operator $F : \mathcal{S}(n)\rightarrow \R$ is continuous and uniformly $(\lambda,\Lambda)$-elliptic in the sense that
	\begin{align*}
		\lambda \text{tr}(N) \leq F(M+N)-F(M) \leq \Lambda \text{tr}(N)
	\end{align*}
	holds with some constants $0<\lambda\leq \Lambda$ and $F(0)=0$, whenever $M,N\in \mathcal{S}(n)$ with $N\geq 0$, where we denote by $\mathcal{S}(n)$ to mean the set of $n\times n$ real symmetric matrices.
	\item[(A2)\label{a2}] $\Phi : \Omega\times [0,\infty)\rightarrow [0,\infty) $ is a continuous map satisfying the following properties: 
	\begin{enumerate}
		\item[1.] There exist constants $ s(\Phi)\geq i(\Phi)>-1$ such that the map $\textstyle t\mapsto \frac{\Phi(x,t)}{t^{i(\Phi)}}$ is almost non-decreasing with constant $L\geq 1$ in $(0,\infty)$ and the map $\textstyle t\mapsto \frac{\Phi(x,t)}{t^{s(\Phi)}}$ is almost non-increasing with constant $L\geq 1$ in $(0,\infty)$ for all $x\in \Omega$.
		\item[2.] There exists constants $0<\nu_0\leq \nu_1$ such that $\displaystyle \nu_{0} \leq  \Phi(x,1) \leq \nu_{1}$ for all $x\in \Omega$.
	\end{enumerate}
	\item[(A3)\label{a3}] $f\in C(\Omega) \cap L^{\infty}({\O})$ and $g\in C^{1,\beta_{g}}(\partial\O)$ for some $\beta_{g}\in (0,1)$.
	\item[(A4)\label{a4}] $\Omega\subset\R^{n}$ is a bounded $C^{1,1}$-domain. 
\end{description}

The Pucci extremal operators $P_{\lambda,\Lambda}^{\pm} : \mathcal{S}(n)\rightarrow \R$ are defined as 
\begin{align*}
	P_{\lambda,\Lambda}^{+}(M):= \Lambda\sum\limits_{\lambda_{k}>0} \lambda_{k} + \lambda\sum\limits_{\lambda_{k}<0}\lambda_{k}
\end{align*}
and
\begin{align*}
	P_{\lambda,\Lambda}^{-}(M):= \lambda\sum\limits_{\lambda_{k}>0} \lambda_{k} +\Lambda\sum\limits_{\lambda_{k}<0}\lambda_{k},
\end{align*}
where $\{\lambda_{k}\}_{k=1}^{n}$ are the eigenvalues of the matrix $M$. The $(\lambda,\Lambda)$-ellipticity of the operator $F$ via the Pucci extremal operators can be formulated as 
\begin{align*}
	P_{\lambda,\Lambda}^{-}(N) \leq F(M+N)-F(M) \leq P_{\lambda,\Lambda}^{+}(N)
\end{align*}
for all $M, N\in \mathcal{S}(n)$.

Before we proceed, we briefly explain several useful results concerning the assumption \ref{a4}. We may assume that $0 \in \partial \Omega$, and  there exist a ball $B=B_R(0)$ in $\mathbb{R}^n$ and $\phi \in C^{1,1}(\mathbb{R}^{n-1})$ such that $\phi(0)=0, \nabla \phi(0)=0$ and
		\begin{align*}
		\Omega \cap B \subset \{y\in B : y_n>\phi(y') \}, \quad \partial \Omega \cap B=\{y\in B : y_n=\phi(y') \}.
		\end{align*}

		\begin{definition}[The ball condition, \cite{AKSZ1}]
			Let $\Omega$ be a bounded domain in $\mathbb{R}^n$. We say that $D$ satisfies the \textit{exterior ball condition} (with radius $r$) if there exists $r>0$ satisfying the following condition: for every $x \in \partial \Omega$, there exists a point $x^e \in \mathbb{R}^n \setminus \Omega$ such that $B_{r}(x^e) \subset \mathbb{R}^n \setminus \Omega$ and $x \in \partial B_{r}(x^e)$. Similarly, we can define the \textit{interior ball condition}. Finally, we say that $\Omega$ satisfies the \textit{ball condition} (with radius $r$) if $\Omega$ satisfies both the exterior and the interior ball condition (with radius $r$).
		\end{definition}
	
 \begin{lemma}[\cite{AKSZ1}]\label{ballcond}
			Let $\Omega \subset \mathbb{R}^n$ be a bounded domain. Then $\Omega$ is a $C^{1, 1}$-domain if and only if $\Omega$ satisfies the ball condition.
		\end{lemma}

On the other hand, for any vector $\xi\in\R^{n}$, we  consider a map $G_{\xi} : \Omega\times \R^{n}\times \mathcal{S}(n)\rightarrow \R$ defined by
\begin{align*}
	G_{\xi}(x,p,M): = \Phi(x,|\xi+p|)F(M)-f(x)
\end{align*}
under the assumptions prescribed in \ref{a1}-\ref{a3}. In \Cref{C01} and \Cref{C1alpha}, we shall focus on viscosity solutions of the equation
\begin{align}\label{xi-eq}
	G_{\xi}(x,Du,D^2u)=0 \text{ in } \Omega.
\end{align}

We now give the definition of a viscosity solution $u$ of the equation \eqref{xi-eq} as follows.

\begin{definition}
	A lower semicontinuous function $v$ is called a \textit{viscosity supersolution} of \eqref{xi-eq} if for all $x_0\in \Omega$ and $\varphi\in C^{2}(\Omega)$ such that $v-\varphi$ has a local minimum at $x_0$, then 
	\begin{align*}
		G_{\xi}(x_{0},D\varphi(x_0),D^{2}\varphi(x_0))\leq 0.
	\end{align*}
	An upper semicontinuous function $w$ is called is a \textit{viscosity subsolution} of \eqref{xi-eq} if for all $x_0\in \Omega$ and $\varphi\in C^{2}(\Omega)$ such that $w-\varphi$ has a local maximum at $x_0$, there holds 
	\begin{align*}
		G_{\xi}(x_{0},D\varphi(x_0),D^{2}\varphi(x_0))\geq 0.
	\end{align*}
	We say that $u\in C(\Omega)$ is a \textit{viscosity solution} of \eqref{xi-eq} if $u$ is a viscosity supersolution and a subsolution simultaneously.
\end{definition}

We also recall a concept of superjet and subjet introduced in \cite{CIL1}.
\begin{definition}
	\label{def:ssjet}
	Let $v : \Omega\rightarrow \R$ be an upper semicontinuos function and $w : \Omega\rightarrow \R$ be a lower semicontinuous function. For every $x_0\in\O$, we define the \textit{second order superjet} of $v$ at $x_0$ by
\begin{equation*}
\begin{array}{rcl}
J^{2,+} v(x_0) & := & \Bigg\{ (p,M)\in \R^{n}\times\mathcal{S}(n) : v(x) \leq
		v(x_0) + \inner{p}{x-x_0} \Bigg.\\
& & \quad\left.+ \frac{1}{2}\inner{M(x-x_0)}{x-x_0} + o(|x-x_0|^{2}) \text{ as } x\rightarrow x_0 \right\}
\end{array}
\end{equation*}
and the \textit{second order subjet} of $w$ at $x_0$ by
\begin{equation*}
\begin{array}{rcl}
J^{2,-} w(x_0) & := & \Bigg\{ (p,M)\in \R^{n}\times\mathcal{S}(n) : w(x) \geq 
		w(x_0) + \inner{p}{x-x_0} \Bigg.\\
& & \quad\left.+ \frac{1}{2}\inner{M(x-x_0)}{x-x_0} + o(|x-x_0|^{2}) \text{ as } x\rightarrow x_0 \right\}.
\end{array}
\end{equation*}
\begin{enumerate}[(i)]
\item A couple $(p,M)\in \R^{n}\times \mathcal{S}(n)$ is a \textit{limiting superjet} of $v$ at $x_{0}\in \O$ if there exists a sequence $\{x_k,p_{k},M_{k}\}\rightarrow \{x,p,M\}$ as $k\rightarrow \infty$  in such a way that $(p_{k},M_{k})\in J^{2,+}v(x_k)$ and $\lim\limits_{k\to\infty}v(x_{k}) = v(x_0)$.
\item A couple $(p,M)\in \R^{n}\times \mathcal{S}(n)$ is a \textit{limiting subjet} of $w$ at $x\in B_{1}$ if there exists a sequence $\{x_k,p_{k},M_{k}\}\rightarrow \{x,p,M\}$ as $k\rightarrow \infty$  in such a way that $(p_{k},M_{k})\in J^{2,-}w(x_k)$ and $\lim\limits_{k\to\infty}w(x_{k}) = w(x_0)$.
	\end{enumerate}
\end{definition}

%

We finish this section by providing the interior $C^{1, \alpha}$-regularity results shown in \cite{BBLL1}. 
\begin{theorem}[{\cite[Theorem 1.1]{BBLL1}}] \label{interior}
	Let $u \in C(B_1)$ be a viscosity solution of
	\begin{align*}
	\Phi(x, |Du|)F(D^2u)=f(x) \quad \text{in} \quad B_1,
	\end{align*}
	under the assumptions \ref{a1} and \ref{a2} with $f \in L^{\infty}(B_1)$. Then $u \in C_{\mathrm{loc}}^{1, \beta}(B_1)$ for all $\beta>0$ satisfying
	\begin{align} \label{beta}
	\beta \in
	\begin{cases}
	(0, \overline{\alpha}) \cap \left(\frac{1}{1+s(\Phi)} \right]   \ &\text{if $i(\Phi) \geq 0$},\\
	(0, \overline{\alpha}) \cap \left(\frac{1}{1+s(\Phi)-i(\Phi)} \right] \ &\text{if $-1<i(\Phi) <0$}.
	\end{cases}
	\end{align} 
 Moreover, for every $\beta$ in \eqref{beta}, there exists a constant $c \equiv c(n, \lambda, \Lambda, i(\Phi), L, \beta)$ such that
	\begin{align*}
	\|u\|_{L^{\infty}(B_{1/2})}+\sup_{x \neq y \in B_{1/2}} \frac{|Du(x)-Du(y)|}{|x-y|^{\beta}} \leq c\left(1+\|u\|_{L^{\infty}(B_1)}+\|f/\nu_0 \|_{L^{\infty}(B_1)}^{\frac{1}{1+i(\Phi)}} \right).
	\end{align*}
\end{theorem}

\section{Alexandroff-Bakelman-Pucci estimate}\label{ABP}
Before we develop $C^{0, 1}$-regularity in \Cref{C01} and $C^{1, \alpha}$-regularity in \Cref{C1alpha}, our study on the global regularity starts with the Alexandroff-Bakelman-Pucci (ABP) estimate. In short, the ABP estimate gives the supremum of $u$ over $\Omega$ in terms of the supremum of $u$ on $\partial \Omega$ and the $L^{n}$-norm of $f$. In this section, we deduce an appropriate version of the ABP estimate for a viscosity subsolution of \eqref{me}. We refer to \cite{DFQ1, Im1, Ju1} for similar results.

\begin{definition}
	\label{def:upconset}
	For $v : \Omega\rightarrow \R$ and $R>0$, the \textit{upper contact set} is defined by 
	\begin{align*}
		\begin{split}
			& \Gamma^{+}(v,\Omega) = \left\{ x\in\Omega : \exists p\in\R^{n}\text{ such that } u(y) \leq
			u(x) + \inner{p}{y-x}\text{ for all }y\in\Omega \right\},
			\\&
			\Gamma^{+}_{R}(v,\Omega) = \left\{ x\in\Omega : \exists p\in\overline{B_{R}(0)} \text{ such that } u(y) \leq
			u(x) + \inner{p}{y-x}\text{ for all }y\in\Omega \right\}.
		\end{split}
	\end{align*}
\end{definition}

We are now ready to prove the ABP estimate.
\begin{theorem}[Alexandroff-Bakelman-Pucci estimate]
	\label{thm_ABP} 
	Suppose that $u\in C(\overline{\O})$ is a viscosity subsolution (resp. supersolution) of \eqref{me} in $\{x\in\Omega : u(x)>0\}$ (resp. $\{x\in\Omega : u(x)<0\}$) under the assumptions \ref{a1}-\ref{a2}. Suppose that $f\in L^{n}(\Omega)\cap C(\Omega)$. Then, there exists a constant $c\equiv c(n,\lambda,i(\Phi),s(\Phi),L,\nu_0)$ such that 
	\begin{align}
		\label{thm_ABP:1}
		\sup\limits_{\Omega} u \leq \sup\limits_{\partial\Omega} g^{+}
		+ c\diam(\Omega)\left(\max\left\{ \norm{f^{-}}^{\frac{1}{i(\Phi)+1}}_{L^{n}(\Gamma^{+}(u^{+}))}, \norm{f^{-}}^{\frac{1}{s(\Phi)+1}}_{L^{n}(\Gamma^{+}(u^{+}))} \right\}+1\right),
	\end{align}
	respectively, 
	\begin{align}
		\label{thm_ABP:2}
		\left(\sup\limits_{\Omega} u^{-}\leq \sup\limits_{\partial\Omega} g^{-}
		+ c\diam(\Omega)\left(\max\left\{ \norm{f^{+}}^{\frac{1}{i(\Phi)+1}}_{L^{n}(\Gamma^{+}(u^{-}))}, \norm{f^{+}}^{\frac{1}{s(\Phi)+1}}_{L^{n}(\Gamma^{+}(u^{-}))} \right\}+1\right)\right).
	\end{align}
	In particular, we have 
	\begin{align}
		\label{thm_ABP:3}
		\norm{u}_{L^{\infty}(\Omega)} \leq
		\norm{g}_{L^{\infty}(\partial\Omega)} + c\diam(\Omega)\left(\max\left\{ \norm{f}_{L^{n}(\Omega)}^{\frac{1}{i(\Phi)+1}}, \norm{f}_{L^{n}(\Omega)}^{\frac{1}{s(\Phi)+1}} \right\}+1\right)
	\end{align}
	for some constant $c\equiv c(n,\lambda,i(\Phi),s(\Phi),L,\nu_0)>0$.
\end{theorem}

\begin{proof}
	The proof consists of two parts. In the first part, we prove the above theorem for the viscosity subsolution $u\in C^{2}(\Omega)\cap C(\overline{\Omega})$. In the second part, we consider $u\in C(\overline{\Omega})$ via approximation based on the sup convolutions. 

\textbf{Part 1.} Suppose that the subsolution $u$ belongs to $C^{2}(\Omega)\cap C(\overline{\Omega})$. Let us define 
	\begin{align*}
		R_0\equiv R_0(u):= \frac{1}{\diam(\Omega)}\left(\sup\limits_{x\in\Omega}u(x)-\sup\limits_{x\in\partial\Omega}u^{+}\right).
	\end{align*}
The purpose is to obtain a certain estimate on $R_0$ in terms of $\norm{f^{-}}_{L^{n}(\Gamma^{+}(u^{+}))}$ and $\data$, from which the estimate \eqref{thm_ABP:1} follows. Applying \cite[Lemma 3.1]{Ju1}, for all $R<R_0$, we find 
\begin{align}
	\label{ABP:2}
	\I_{B_{R}(0)} g(z)\,dz \leq
	\I_{\Gamma_{R}^{+}(u^{+})}g(Du)\left|\det\left(D^{2}u\right)\right|\,dx
	\quad (\forall g\in C(\R^{n}),\, g\geq 0)
\end{align}
and 
\begin{align}
	\label{ABP:3}
	D^{2}u \leq 0 \text{ on } \Gamma_{R}^{+}(u^{+})\subset \{x\in \Omega : u(x)>0\}.
\end{align}

Let us now discuss the behavior of $Du$ in the set $\Gamma_{R}^{+}(u^{+})$. Let $x_0\in \Gamma_{R}^{+}(u^{+})$ be any point. If $Du(x_0)\neq 0$, then we are able to take $u$ as a test function in the definition of viscosity subsolution. In turn, we have 
\begin{align*}
	\Phi(x_0,|Du(x_0)|)F(D^2 u(x_0)) \geq f(x_0).
\end{align*}
Then we see 
\begin{align*}
	-f^{-}(x_0) \leq f(x_0) \leq \Phi(x_0,|Du(x_0)|)F(D^{2}u(x_0))
	\leq
	\Phi(x_0,|Du(x_0)|) P_{\lambda,\Lambda}^{+}(D^2u(x_0)).
\end{align*}
Recalling $D^2u(x_0)\leq 0$ by \eqref{ABP:3}, we find $P_{\lambda,\Lambda}^{+}(D^{2}u(x_0)) = \lambda\tr(D^{2}u(x_0))$ and 
\begin{align}
	\label{ABP:6}
	\left(\frac{-\tr(D^2u(x_0))}{n} \right)^{n} \leq \left(\frac{f^{-}(x_0)}{n\lambda\Phi(x_0,|Du(x_0)|)} \right)^{n}.
\end{align}
If $Du(x_0)=0$ and $D^2 u(x_0)\neq 0$, then $x_0$ is a critical point of $u$. On the other hand, recalling again \eqref{ABP:3}, we have $D^2u(x_0)<0$, which means that $x_0$ is a non-degenerate critical point of $u$. However, the set of non-degenerate critical points of $u$ is countable since $u\in C^2(\Omega)$. 

Let us recall also the following classical inequality,
\begin{align*}
	\det(A)\det(B) \leq \left(\frac{\tr(AB)}{n} \right)^{n}
	\text{ for all } A,B \in \mathcal{S}(n)\text{ with } A,B\geq 0.
\end{align*}
In turn, the last display together with \eqref{ABP:3} and \eqref{ABP:6} yields
\begin{align}
	\label{ABP:8}
	\left|\det D^{2}u(x) \right| \leq \left( \frac{f^{-}(x)}{n\lambda\Phi(x,|Du(x)|)} \right)^{n}
\end{align}
for all $x\in \Gamma_{R}^{+}(u^{+})\setminus\mathcal{U}$, where $\mathcal{U} = \{x\in \Gamma_{R}^{+}(u^{+}) : Du(x)=0\}$. Now we consider two steps depending on the sign of $i(\Phi)$.

\textbf{Step 1. $i(\Phi)\geq 0$:} Let us select $g(z) = \min\left\{|z|^{i(\Phi)n}, |z|^{s(\Phi)n}\right\}$ in \eqref{ABP:2}. In turn, recalling \eqref{ABP:8}, we find 
\begin{align*}
	I_{1}&:=\I_{B_{R}(0)}\min\left\{|z|^{i(\Phi)n}, |z|^{s(\Phi)n}\right\}\,dz
	\notag\\&
	\leq
	\I_{\Gamma_{R}^{+}(u^{+})\setminus \mathcal{U}}\min\left\{|Du|^{i(\Phi)n}, |Du|^{s(\Phi)n}\right\}\left( \frac{f^{-}}{n\lambda\Phi(x,|Du|)} \right)^{n}\,dx
	\notag\\&
	\leq
	\frac{L^{n}}{n^{n}\lambda^{n}\nu_{0}^{n}}\I_{\Gamma_{R}^{+}(u^{+})} (f^{-})^{n}\,dx,
\end{align*}
where we have used \ref{a2}. On the other hand, by co-area formula, we have 
\begin{align*}
	I_{1} &= \I_{0}^{R}\min\left\{t^{i(\Phi)n}, t^{s(\Phi)n}\right\}\I_{\partial B_{t}(0)}\,dSdt
	= n \omega_{n}\I_{0}^{R}\min\left\{t^{i(\Phi)n}, t^{s(\Phi)n}\right\}t^{n-1}\,dt
	\notag\\&
	=
	\left\{\begin{array}{lr}
        \frac{\omega_{n}R^{(i(\Phi)+1)n}}{i(\Phi)+1}- \frac{\omega_{n}(s(\Phi)-i(\Phi))}{(i(\Phi)+1)(s(\Phi)+1)} & \text{if } R\geq 1,\\
        \frac{\omega_{n}R^{(s(\Phi)+1)n}}{s(\Phi)+1} & \text{if } R < 1.
        \end{array}\right.
\end{align*}
Combining the last two displays, we arrive at \eqref{thm_ABP:1}.

\textbf{Step 2: $-1<i(\Phi)<0$.} In this case, we select $g(z) = \left( \frac{|z|}{|z|+\delta} \right)^{-i(\Phi)n}\min\left\{ |z|^{i(\Phi)n},|z|^{s(\Phi)n} \right\}$ for an arbitrary  number $\delta>0$. Clearly, $g\in C(\R^{n})$ and so we have
\begin{align*}
	I_{2}(\delta)&:= \I_{B_{R}(0)}\left( \frac{|z|}{|z|+\delta} \right)^{-i(\Phi)n}\min\left\{ |z|^{i(\Phi)n},|z|^{s(\Phi)n} \right\}\,dz
	\notag\\&
	\leq
	\I_{\Gamma_{R}^{+}(u^{+})\setminus\mathcal{U}}\left( \frac{|Du|}{|Du|+\delta} \right)^{-i(\Phi)n}\min\left\{ |Du|^{i(\Phi)n},|Du|^{s(\Phi)n} \right\}\left(\frac{f^{-}}{n\lambda\Phi(x,|Du|)} \right)^{n}\,dx
	\notag\\&
	\leq
	\frac{L^{n}}{n^{n}\lambda^{n}\nu_{0}^{n}}\I_{\Gamma_{R}^{+}(u^{+})}(f^{-})^{n}\,dx,
\end{align*}
where we have used again \ref{a2} and the fact that $i(\Phi)<0$. By using co-area formula and recalling that $-1<i(\Phi)<0$, we get 
\begin{align*}
	I_{2}(\delta)&=
	 n \omega_{n}\I_{0}^{R}\left(\frac{t}{t+\delta} \right)^{-i(\Phi)n} \min\left\{ t^{i(\Phi)n},t^{s(\Phi)n} \right\}t^{n-1}\,dt.
\end{align*}
By applying Lebesgue's dominated convergence theorem, we conclude
\begin{align*}
\lim_{\delta \to 0^{+}}I_{2}(\delta)=&= n \omega_{n}\I_{0}^{R}\min\left\{t^{i(\Phi)n}, t^{s(\Phi)n}\right\}t^{n-1}\,dt
	\notag\\&
	=
	\left\{\begin{array}{lr}
        \frac{\omega_{n}R^{(i(\Phi)+1)n}}{i(\Phi)+1}- \frac{\omega_{n}(s(\Phi)-i(\Phi))}{(i(\Phi)+1)(s(\Phi)+1)} & \text{if } R\geq 1,\\
        \frac{\omega_{n}R^{(s(\Phi)+1)n}}{s(\Phi)+1} & \text{if } R < 1.
        \end{array}\right.
\end{align*}

Combining the last two displays, we get \eqref{thm_ABP:1}.

\textbf{Part 2.} Let $u\in C(\overline{\Omega})$. Since we have ABP estimates for $u\in C^{2}(\Omega)\cap C(\overline{\Omega})$, the remainder of the proof can be argued similarly as in the proof of \cite[Theorem 1.1]{Ju1}, see also \cite{DFQ1}.
\end{proof}


\section{Local Lipschitz estimates up to the boundary}\label{C01}
By \Cref{thm_ABP}, any viscosity solution of \eqref{me} is bounded in $L^{\infty}(\Omega)$ under the assumptions \ref{a1}-\ref{a3}. In this section, to derive further Lipschitz estimates up to the boundary as in \cite{BD1, dSJRR1}, we consider a bounded viscosity solution of 
\begin{align} \label{locDP} 
\left\{ \begin{array}{ll} 
\Phi(y, |Du|) F(D^2u)=f(y) & \text{in $B \cap \{y_n > \phi(y')\}$},\\
u(y)=g(y) & \text{on $B \cap \{y_n=\phi(y')\}$},
\end{array} \right.
\end{align}
where the function $\phi$ is introduced in \Cref{preliminaries}.

\begin{remark}[Smallness regime]\label{rmk:small}
Here we verify that, for a bounded viscosity solution $u$ of 
\begin{align}\label{sr:0}
\left\{ \begin{array}{ll} 
\Phi(y, |\xi+Du|) F(D^2u)=f(y) & \text{in $B \cap \{y_n > \phi(y')\}$},\\
u(y)=g(y) & \text{on $B \cap \{y_n=\phi(y')\}$},
\end{array} \right.
\end{align} 
we are able to assume 
\begin{align}
	\label{small}
	\|u\|_{L^{\infty}(B_{1} \cap \{y_n > \phi(y')\})} \leq 1, \, \|g\|_{C^{1, \beta_g}(B_1 \cap \{y_n = \phi(y')\})} \leq 1  \text{ and }\|f\|_{L^{\infty}(B_{1} \cap \{y_n > \phi(y')\})} \leq \varepsilon_{0},
\end{align}
for some constant $0<\varepsilon_{0}<1$ small enough, and also $\nu_0=\nu_1=1$ in \ref{a2}.
In order to consider the problem in a smallness regime as in \eqref{small}, for a fixed ball $B_{r}(x)\subset B$, we define $\bar{u} : B_{1} \cap \{y_{n} >\bar{\phi}(y')\} \rightarrow \R$ by
\begin{align*}
	\bar{u}(y):= \frac{u(ry+x)}{K}
\end{align*}
for a function $\bar{\phi}$ and positive constants $K\geq 1\geq r$ to be determined later. It can be seen that $\bar{u}$ is a viscosity solution of 
\begin{align} \label{sr:2} 
\left\{ \begin{array}{ll} 
\bar{\Phi}(y,|\bar{\xi}+D\bar{u}|)\bar{F}(D^{2}\bar{u})=\bar{f}(y) & \text{in $B_{1} \cap \{y_n > \bar{\phi}(y')\}$},\\
\bar{u}(y)=\bar{g}(y) & \text{on $B_{1} \cap \{y_n=\bar{\phi}(y')\}$},
\end{array} \right.
\end{align}
where 
\begin{align*}
	\displaystyle
	\bar{F}(M)&:= \frac{r^{2}}{K}F\left(\frac{K}{r^{2}}M \right),\quad
	\bar{\Phi}(y,t) := \frac{\Phi\left(ry+x,\frac{K}{r}t\right)}{\Phi\left(ry+x,\frac{K}{r}\right)},\quad
	\bar{f}(y) := \frac{r^{2}}{\Phi\left(ry+x,\frac{K}{r}\right)K}f(ry+x),\\
	 \bar{\xi}&:= \frac{r}{K}\xi,\quad
	 \bar{\phi}(y'):=\frac{\phi(ry'+x')-x_{n}}{r} \text{ and }
	 \bar{g}(y):=\frac{g(ry+x)}{K}.
\end{align*}
Note that $\bar{F}$ is still a uniformly $(\lambda,\Lambda)$-elliptic operator, the map $\displaystyle t\mapsto \frac{\bar{\Phi}(y,t)}{t^{i(\Phi)}}$ is almost non-decreasing and the map $\displaystyle t\mapsto \frac{\bar{\Phi}(y,t)}{t^{s(\Phi)}}$ is almost non-increasing with the same constants $L\geq 1$ and $s(\Phi)\geq i(\Phi)>-1$ as in \ref{a2}, and $\bar{\Phi}(y,1)=1$ for all $y\in B_{1}$. It is immediate from the choice of $r$ that $\|D^{2}\bar{\phi}\|_{\infty} \leq \|D^{2}\phi\|_{\infty}$ and
\begin{align*}
 \|\bar{g}\|_{C^{1, \beta_g}(B_1 \cap \{y_n = \bar{\phi}(y')\})} \leq \frac{1}{K}  \|g\|_{C^{1, \beta_{g}}(\partial \Omega)}.
\end{align*}
Moreover, the assumption \ref{a2} implies
\begin{align*}
	\norm{\bar{f}}_{L^{\infty}(B_{1} \cap \{y_n > \bar{\phi}(y')\})}
	\leq \frac{L r^{2+i(\Phi)}}{\nu_{0} K^{1+i(\Phi)}}\norm{f}_{L^{\infty}(\Omega)}.
\end{align*}
By recalling $i(\Phi)> -1$ and setting 
\begin{align*}
	K:= 2\left(1+\norm{u}_{L^{\infty}(\Omega)} + \|g\|_{C^{1, \beta_{g}}(\partial \Omega)} + \left[\frac{L}{\nu_0}\norm{f}_{L^{\infty}(\Omega)} \right]^{\frac{1}{1+i(\Phi)}}\right)
\end{align*}
and 
\begin{align*}
	r:= \varepsilon_{0}^{\frac{1}{2+i(\Phi)}},
\end{align*}
we see that $\bar{u}$ solves the equation \eqref{sr:2} under the smallness regime in \eqref{small}.
\end{remark}

If we have special conditions on a viscosity supersolution (or subsolution), then we can apply the comparison principle without an additional structure condition such as \ref{a5} in \Cref{Perron}. See \Cref{Perron} for more comments on the comparison principle.

\begin{lemma}[Comparison principle I] \label{specialCP}
	Let $f_{1}, f_{2} \in C(\overline{\Omega})$ with $f_{1} >f_{2}$ and $v \in C(\overline{\Omega})$ be a viscosity subsolution of $\Phi(y, |Du|)F(D^{2}u)=f_{1}(y)$ in $\Omega$. Moreover, let $w \in C(\overline{\Omega}) \cap C^{2}(\Omega)$ be a viscosity supersolution of $\Phi(y, |Du|)F(D^{2}u)=f_{2}(y)$. If $v\leq w$ on $\partial\Omega$, then $v\leq w$ in $\Omega$.
\end{lemma}

\begin{proof}
	By contradiction, we suppose that
	\begin{align*}
	 	\max_{x \in \overline{\Omega}} (v(x)-w(x)) >0
	\end{align*}
	and the maximum is achieved at a point $\hat{x} \in \Omega$. Since $v$ is a viscosity subsolution, $w \in C^{2}$ and $v-w$ has a local maximum at $\hat{x}$, the definition of viscosity subsolutions  yields
	\begin{align*}
	 	\Phi(\hat{x}, |Dw(\hat{x})|)F(D^{2}w(\hat{x})) \geq f_{1}(\hat{x}).
	\end{align*}
	On the other hand, since $w$ is a viscosity supersolution, we have
	\begin{align*}
	\Phi(\hat{x}, |Dw(\hat{x})|)F(D^{2}w(\hat{x})) \leq f_{2}(\hat{x}),
	\end{align*}
	which leads to the contradiction.
\end{proof}

\begin{lemma}\label{distlem}
	Let $g \in C^{1, \beta_g}(\partial \Omega)$. Let $d$ be the distance to the hypersurface $\{y_n=\phi(y')\}$.\\
	Then for every $r \in (0, 1)$ and $\gamma \in (0, 1)$, there exists $\delta_0>0$ depending on $\|f\|_{L^{\infty}(B_{1} \cap \{y_n > \phi(y')\})}$, $\lambda$, $\Lambda$, $\Omega$, $r$, $L$, $\nu_0$ and $\mathrm{Lip}_g(\partial \Omega)$ such that for every $0<\delta<\delta_0$, if $u$ is a viscosity solution of \eqref{locDP}
	with $\|u\|_{L^{\infty}(B_{1} \cap \{y_n > \phi(y')\})}\leq 1$, then
	\begin{align*}
	|u(y', y_n)-g(y')| \leq \frac{6}{\delta} \frac{d(y)}{1+d(y)^{\gamma}} \quad \text{in $B_r(0) \cap \{y_n>\phi(y')\}$.}
	\end{align*}
\end{lemma}

\begin{proof}
 We separate two cases: (i) $g \equiv 0$, (ii) $g$ is not identically zero.
	\begin{enumerate}[(i)]
		\item ($g \equiv 0$) In this case, we have $\|u\|_{L^{\infty}} \leq 1$ and so we will only consider the smaller set $\Omega_{\delta}:=\{y \in \Omega : d(y) <\delta\}$. Moreover, we choose $\delta_1>0$ such that if $d(y) <\delta_1$, then $d$ belongs to $C^2$ and $|D^2d| \leq K$ for some universal constant $K>0$.
		
		The proof relies on the construction of upper and lower barriers. For this purpose, we define a function $w \in C^{2}(\Omega_{\delta})$ by
		\begin{align*}
		w(y)=
		\left\{ \begin{array}{ll} 
		\frac{2}{\delta}\frac{d(y)}{1+d^{\gamma}(y)} & \text{for $|y|<r$},\\
		\frac{2}{\delta}\frac{d(y)}{1+d^{\gamma}(y)}+\frac{1}{(1-r)^3}(|y|-r)^3 & \text{for $|y| \geq r$}.
		\end{array} \right.
		\end{align*}
		By following the argument in \cite[Lemma 2.2]{BD4}, we have $w \geq u$ on $\partial(B_{1} \cap \{y_n>\phi(y')\} \cap \Omega_{\delta})$. Moreover, it is easily checked that $|Dw| \geq \frac{1}{4\delta}$ when $\delta \leq \frac{1-r}{12}$ and so if we choose $\delta<1/4$, then $|Dw| \geq 1$. Moreover, we can calculate
		\begin{align*}
			\mathcal{P}_{\lambda, \Lambda}^+(D^2w) \leq -2\gamma \delta^{\gamma-2}\lambda \frac{1+\gamma}{(1+\delta^{\gamma})^3}+\frac{2}{\delta}nK\Lambda+\frac{6n\Lambda}{(1-r)^2} \lesssim -\delta^{\gamma-2}+\delta^{-1}. 
		\end{align*}
		Since $\gamma-2<-1<0$, we can further choose $\delta \in (0,1)$ small enough so that $\mathcal{P}_{\lambda, \Lambda}^+(D^2w)<0$. Then, by recalling \ref{a2}, 
		\begin{align*}
		\Phi(x, |Dw|) \mathcal{P}_{\lambda, \Lambda}^+(D^2w) \lesssim -L\nu_0|Dw|^{i(\Phi)} (\delta^{\gamma-2}-\delta^{-1}) \leq -L\nu_0 (\delta^{\gamma-i(\Phi)-2}-\delta^{-1-i(\Phi)}).
		\end{align*}
		Since $\gamma-i(\Phi)-2 <-1-i(\Phi)<0$, we finally choose $\delta \in (0,1)$ small enough so that 
		\begin{align*}
		\Phi(x, |Dw|) \mathcal{P}_{\lambda, \Lambda}^+(D^2w) <-\|f\|_{\infty}-1.
		\end{align*}
		By applying \Cref{specialCP}, we conclude that 
		\begin{align*}
		u \leq w=\frac{2}{\delta}\frac{d(y)}{1+d^{\gamma}(y)} \quad \text{in $B_r(0) \cap \{y_n >\phi(y')\}$.}
		\end{align*}
		The lower bound for $u$ can be obtained in a similar argument.

		\item ($g$ is not identically zero) This case follows from the same argument as in \cite[Lemma 2.2]{BD4}.
	\end{enumerate}
	
\end{proof}

The main theorem in this section is the following boundary Lipschitz estimate, whose proof relies on \Cref{distlem} and the Ishii-Lions Lemma \cite{CIL1}.
\begin{theorem}[Lipschitz estimates for $\xi=0$]\label{basicreg}
	Let $g$ be a Lipschitz continuous function. Suppose that $u$ satisfies \eqref{locDP} with $\|u\|_{L^{\infty}(B_{1} \cap \{y_n > \phi(y')\})} \leq 1$. Then for every $r \in (0, 1)$, we have $u \in C^{0, 1}(B_r \cap \{y_n >\phi(y')\})$ and
	\begin{align} \label{basicest}
	\|u\|_{C^{0, 1}(B_r \cap \{y_n >\phi(y')\})} \leq C(n, \lambda, \Lambda, i(\Phi), s(\Phi), r, L, \mathrm{Lip}_g(\partial \Omega), \|f\|_{L^{\infty}(\Omega)}).
	\end{align}
\end{theorem}

\begin{proof}
	Let $r_1 \in (r, 1)$ be fixed. For $x_0 \in B_r \cap \{y_n > \phi(y')\}$, we define
	\begin{align*}
		\Psi(x, y)\coloneqq u(x)-u(y)-M\omega(|x-y|)-L\left(|x-x_0|^2+|y-x_0|^2\right),
	\end{align*}
	where
	\begin{align*}
	\omega(s):=
	\begin{cases}
	s-\omega_0 s^{\frac{3}{2}} \ & \text{if $s \leq s_0:=\left(\frac{2}{3\omega_0}\right)^2$},\\
	\omega(s_0) \ & \text{if $s \geq s_0$.}
	\end{cases}
	\end{align*}
	We claim that for $L, M \gg 1$ large enough, 
	\begin{align}\label{ishiilions}
	\Psi(x, y) \leq 0 \quad \text{for all $(x, y) \in (B_{r_1} \cap \overline{\Omega}) \times (B_{r_1} \cap \overline{\Omega})$}.
	\end{align}
	Note that this inequality implies the desired Lipschitz estimate. 
	
	First of all, suppose that $y \in B_{r_1} \cap \{y_n=\phi(y')\}$. Then by \Cref{distlem}, there exists a constant $K_0>0$ such that
	\begin{align*}
	|u(z)-g(z')| \leq K_0 d(z, \partial \Omega) \quad \text{for $z \in B_{r_1} \cap \{y_n>\phi(y')\}$},
	\end{align*}
	which implies that 
	\begin{align*}
		|u(x)-u(y)| &\leq |u(x', x_n)-u(x', \phi(x'))|+|u(x', \phi(x'))-u(y', \phi(y'))| \\
		&\leq K_0 d(x, \partial \Omega)+\mathrm{Lip}_g(\partial \Omega)|x'-y'| \leq (K_0+\mathrm{Lip}_g(\partial \Omega))|x-y|.
	\end{align*}
	Therefore, if we choose $M/3 \geq K_0+\mathrm{Lip}_g(\partial \Omega)$, then
	\begin{align*}
	\Psi(x,y) \leq M\left(\frac{|x-y|}{3}-\omega(|x-y|)\right)-L\left(|x-x_0|^2+|y-x_0|^2\right) \leq 0.
	\end{align*}
	We now prove \eqref{ishiilions} by contradiction; suppose that there exists some point $(\hat{x}, \hat{y}) \in (B_{r_1} \cap \overline{\Omega}) \times (B_{r_1} \cap \overline{\Omega})$ such that
	\begin{align*}
	\Psi(\hat{x}, \hat{y})=\max_{(B_{r_1} \cap \overline{\Omega}) \times (B_{r_1} \cap \overline{\Omega})} \Psi(x,y)>0.
	\end{align*}
	Here, we also choose $L>\max\left\{\frac{8}{(r_1-r)^2}, \frac{1}{2(r+r_1)}\right\}$. Then we can easily check that
(i) $\hat{x} \neq \hat{y}$; (ii) $\hat{x}, \hat{y} \in B_{r_1} \cap \{y_n>\phi(y')\}$; (iii) $\hat{x}, \hat{y} \in B_{\frac{r_1+r}{2}}$. Thus, by applying Ishii-Lions Lemma \cite[Theorem 3.2]{CIL1}, we see that, for every $\varepsilon>0$ sufficiently small, there exist $X, Y \in \mathcal{S}^n$ such that
\begin{align}
\label{matrix:1}
(M\omega'(|\hat{x}-\hat{y}|)\hat{a}+2L(\hat{x}-x_0), X) \in \overline{J}^{2, +}u(\hat{x}),
\notag\\(M\omega'(|\hat{x}-\hat{y}|)\hat{a}-2L(\hat{y}-x_0), -Y) \in \overline{J}^{2, -}u(\hat{y}),
\notag\\
\begin{pmatrix}
X & 0 \\
0 & Y
\end{pmatrix}
\leq M
\begin{pmatrix}
Z & -Z \\
-Z & Z
\end{pmatrix}
+(2L+\varepsilon)
\begin{pmatrix}
I & 0 \\
0 & I
\end{pmatrix}
,
\end{align}
where 
\begin{align*}
Z=\omega''(|\hat{x}-\hat{y}|)\hat{a} \otimes \hat{a}+\frac{\omega'(|\hat{x}-\hat{y}|)}{|\hat{x}-\hat{y}|}(I-\hat{a}\otimes\hat{a}) \quad \text{for $\hat{a} \coloneqq \frac{\hat{x}-\hat{y}}{|\hat{x}-\hat{y}|}$.}
\end{align*}

For simplicity, we write $q_x \coloneqq M\omega'(|\hat{x}-\hat{y}|)\hat{a}+2L(\hat{x}-x_0)$ and $q_y \coloneqq M\omega'(|\hat{x}-\hat{y}|)\hat{a}-2L(\hat{y}-x_0)$. We first choose $\omega_0>0$ small enough so that $s_0 \geq 2 > r+r_1$. Note that $t \mapsto \omega'(t)$ is decreasing on $t \in [0, s_0]$. If we choose $M>0$ large enough so that $2L(r+r_1) \leq  \frac{M}{2}\omega'(r+r_1)$, then we have $2L|\hat{x}-x_0|, 2L|\hat{y}-x_0| \leq \frac{M}{2}\omega'(|\hat{x}-\hat{y}|)$. In particular, we obtain 
\begin{align} \label{estq}
\frac{M}{2}\omega'(|\hat{x}-\hat{y}|) \leq |q_x|, |q_y| \leq {2M} \omega'(|\hat{x}-\hat{y}|),
\end{align}
and by the choice of $L$, we also know that $|q_x|, |q_y| \geq 1$. 

On the other hand, for $X$ and $Y$, we will use the matrix inequality \eqref{matrix:1}. First, by evaluating a vector of the form $(\xi, \xi)$ for any $\xi \in \mathbb{R}^n$, we have
\begin{align*}
(X+Y)\xi \cdot \xi \leq 6L|\xi|^2,
\end{align*}
which implies that any eigenvalues of $X+Y$ are less than $6L$. Moreover, by applying the matrix inequality \eqref{matrix:1} for $(\hat{a}, -\hat{a})$, we observe
\begin{align*}
(X+Y)\hat{a} \cdot \hat{a} \leq 4M\omega''(|\hat{x}-\hat{y}|)+6L=-3M\omega_0 |\hat{x}-\hat{y}|^{-1/2}+6L.
\end{align*}
In other words, at least one eigenvalue of $X+Y$ is less than $-3M\omega_0 |\hat{x}-\hat{y}|^{-1/2}+6L$. Therefore, by the definition of the Pucci operator, we have
\begin{align*}
\mathcal{P}_{\lambda, \Lambda}^+(X+Y) &\leq \lambda(-3M\omega_0 |\hat{x}-\hat{y}|^{-1/2}+6L)+6\Lambda(n-1)L \\
&= -3\lambda M\omega_0|\hat{x}-\hat{y}|^{-1/2}+6[\Lambda(n-1)+\lambda]L.
\end{align*}
We now employ the definition of limiting superjet and limiting subjet:
\begin{align*}
&\Phi(\hat{x}, |q_x|) F(X) \geq f(\hat{x}) \geq -\|f\|_{\infty}, \\
&\Phi(\hat{y}, |q_y|) F(-Y) \leq f(\hat{y}) \leq \|f\|_{\infty}.
\end{align*}
Since $|q_x|, |q_y| \geq 1$, an application of \ref{a2} and \eqref{estq} yields that
\begin{align*}
\Phi(\hat{x}, |q_x|) \geq L\nu_0 |q_x|^{i(\Phi)}, \quad  
\Phi(\hat{y}, |q_y|) \geq L\nu_0 |q_y|^{i(\Phi)}.
\end{align*}
Moreover, \ref{a1} shows that
\begin{align*}
F(X)-F(-Y) \leq \mathcal{P}^+_{\lambda, \Lambda}(X+Y) \leq -3\lambda M\omega_0|\hat{x}-\hat{y}|^{-1/2}+6[\Lambda(n-1)+\lambda]L.
\end{align*} 
Combining these results, we have
\begin{align*}
-\|f\|_{\infty} \left(|q_x|^{-i(\Phi)}+|q_y|^{-i(\Phi)}\right) \leq -3\lambda M\omega_0|\hat{x}-\hat{y}|^{-1/2}+6[\Lambda(n-1)+\lambda]L.
\end{align*}
We now split into two cases depending on the sign of $i(\Phi)$:
\begin{enumerate}[(i)]
	\item ($i(\Phi) \geq 0$) Since $|q_x|, |q_y| \geq 1$ and $|\hat{x}-\hat{y}|\leq 1$, we conclude that
	\begin{align*}
	3\lambda M\omega_0 \leq 2\|f\|_{\infty}+6[\Lambda(n-1)+\lambda]L,
	\end{align*}
	which does not hold for sufficiently large $M>0$.
	
	\item ($-1<i(\Phi)<0$) Recalling that $|q_x|, |q_y| \leq {2 M}\omega'(|\hat{x}-\hat{y}|) \leq 2M$, we derive
	\begin{align*}
	3\lambda M \omega_0 \leq 2\|f\|_{\infty} (2M)^{-i(\Phi)}+6[\Lambda(n-1)+\lambda]L.
	\end{align*}
	Since $-i(\Phi)<1$, this inequality does not hold for sufficiently large $M>0$.
\end{enumerate}
This finishes the proof for \eqref{ishiilions}.
\end{proof}

On the other hand, for a modified equation \eqref{xi-eq}, we can prove the boundary Lipschitz esimate, provided that $|\xi|$ is large. In short, the boundary Lipschitz estimates hold when either 
\begin{enumerate}[(i)]
	\item $\xi=0$ with $i(\Phi)>-1$ (\Cref{distlem} and \Cref{basicreg}) or
	
	\item $|\xi|$ is large with $i(\Phi) \geq 0$ (\Cref{distlem2} and \Cref{basicreg2}). 
\end{enumerate}

\begin{lemma}\label{distlem2}	
	Let $g$ be Lipschitz continuous on $\partial \Omega$ and $\xi \in \mathbb{R}^n$ with $|\xi|=1$. Then for every $r \in (0, 1)$ and $\gamma \in (0, 1)$, there exists $\delta>0$ depending on $\lambda$, $\Lambda$, $s(\Phi)$, $r$ and $\mathrm{Lip}_g(\partial \Omega)$ such that for $0\leq b<\frac{\delta}{6}$, any viscosity solution $u$ of 
	\begin{align} \label{modeq1}
	\left\{ \begin{array}{ll} 
	\Phi(y, |\xi+bDu|) F(D^2u)=f(y) & \text{in $B_1 \cap \{y_n > \phi(y')\}$},\\
	u(y)=g(y) & \text{on $B_1 \cap \{y_n = \phi(y')\}$},
	\end{array} \right.
	\end{align}
	with 
	\begin{align*}
	\|u\|_{L^{\infty}(B_1 \cap \{y_n > \phi(y')\})}\leq 1 \quad \text{and} \quad \|f\|_{L^{\infty}(B_1 \cap \{y_n > \phi(y')\})}\leq \varepsilon_0
	\end{align*}
	satisfies
	\begin{align*}
	|u(y', y_n)-g(y')| \leq \frac{6}{\delta} \frac{d(y)}{1+d(y)^{\gamma}} \quad \text{in $B_r \cap \{y_n > \phi(y')\}$.}
	\end{align*}
\end{lemma}

\begin{proof}
	As in the proof of \Cref{distlem}, we may suppose that $g \equiv 0$ and construct a barrier function in a local domain $\Omega_{\delta}:=\{y \in \Omega : d(y) <\delta\}$. If $b=0$, then there is no degeneracy with respect to the gradient $Du$ and so the result holds. Thus, we also suppose that $b > 0$.
	
	We now define a function $w \in C^{2}(\Omega_{\delta})$ by
		\begin{align*}
w(y)=
\left\{ \begin{array}{ll} 
\frac{2}{\delta}\frac{d(y)}{1+d^{\gamma}(y)} & \text{for $|y|<r$},\\
\frac{2}{\delta}\frac{d(y)}{1+d^{\gamma}(y)}+\frac{1}{(1-r)^3}(|y|-r)^3 & \text{for $|y| \geq r$}.
\end{array} \right.
\end{align*}
	We recall that $w \geq u$ on $\partial(B \cap \{y_n>\phi(y')\} \cap \Omega_{\delta})$ and
	\begin{align*}
	\mathcal{P}_{\lambda, \Lambda}^+(D^2w) \leq -2\gamma \delta^{\gamma-2}\lambda \frac{1+\gamma}{(1+\delta^{\gamma})^3}+\frac{2}{\delta}nK\Lambda+\frac{6n\Lambda}{(1-r)^2} \lesssim -\delta^{\gamma-2}+\delta^{-1}. 
	\end{align*}
	On the other hand, since
	\begin{align*}
	Dw(y)=
	\left\{ \begin{array}{ll} 
	\frac{2}{\delta}\frac{1+(1-\gamma)d^{\gamma}}{(1+d^{\gamma})^2} Dd & \text{for $|y|<r$},\\
	\frac{2}{\delta}\frac{1+(1-\gamma)d^{\gamma}}{(1+d^{\gamma})^2}Dd+\frac{y}{|y|}\frac{3}{(1-r)^3}(|y|-r)^2 & \text{for $|y| \geq r$},
	\end{array} \right.
	\end{align*}
	we have $|Dw| \leq \frac{3}{\delta}$ provided that $\delta \leq \frac{1-r}{3}$. As a consequence, we derive
	\begin{align*}
		\frac{1}{2} \leq |\xi+bDw| \leq \frac{3}{2}  \quad \text{for $0<b<\frac{\delta}{6}$},
	\end{align*}
	and so we conclude that
	\begin{align*}
		\Phi(y, |\xi+bDw|)F(D^2w) <-\|f\|_{\infty}-1, \quad \text{for sufficiently small $\delta>0$.}
	\end{align*}
	\Cref{specialCP} yields the upper bound for $u$, and the remaining part can be done as in \Cref{distlem}.
\end{proof}

Note that \Cref{distlem2} holds for any $i(\Phi)>-1$, while \Cref{basicreg2} holds only for the degenerate case, $i(\Phi) \geq 0$.
\begin{theorem}[Lipschitz estimates for large $|\xi|$ with $i(\Phi) \geq 0$]\label{basicreg2}
	Let $g$ be Lipschitz continuous on $\partial \Omega$ and $\xi \in \mathbb{R}^n$. Assume that $u$ is a viscosity solution of 
	\begin{align} \label{modeq2}
	\left\{ \begin{array}{ll} 
	\Phi(y, |\xi+Du|) F(D^2u)=f(y) & \text{in $B_1 \cap \{y_n > \phi(y')\}$},\\
	u(y)=g(y) & \text{on $B_1 \cap \{y_n=\phi(y')\}$},
	\end{array} \right.
	\end{align}
	with 
	\begin{align*}
	\|u\|_{L^{\infty}(B_1 \cap \{y_n > \phi(y')\})}\leq 1 \quad \text{and} \quad \|f\|_{L^{\infty}(B_1 \cap \{y_n > \phi(y')\})}\leq \varepsilon_0.
	\end{align*}
	Then for all $r \in (0, 1)$, there exists $p_0=p_0(\lambda, \Lambda, n, i(\Phi), s(\Phi), r, \varepsilon_0, \mathrm{Lip}_g(\partial \Omega))>0$, such that if $|\xi| > p_0$, then $u \in C^{0, 1}(B_r \cap \{y_n > \phi(y')\})$ and we have the estimate
	\begin{align*}
	\|u\|_{C^{0, 1}(B_r \cap \{y_n > \phi(y')\})} \leq C(\lambda, \Lambda, n, i(\Phi), s(\Phi), r, \varepsilon_0, \mathrm{Lip}_g(\partial \Omega)).
	\end{align*}
\end{theorem}

\begin{proof}
	Since the proof is similar to the one of \Cref{basicreg}, here we concentrate on the differences.
	\begin{enumerate}[(i)]
	\item We first need to show that if $x$ or $y$ belongs to $B_{r_{1}} \cap \{y_{n}=\phi(y')\}$, then $\Psi(x, y) \leq 0$. In the case of \Cref{basicreg}, this result immediately followed from \Cref{distlem}. In a similar manner,  it is enough to apply \Cref{distlem2} for a solution $u$. More precisely, if $u$ is a solution of \eqref{modeq2}, then $u$ solves	
	\begin{align*}
	\left\{ \begin{array}{ll} 
	\tilde{\Phi}\left(y, \left|{\xi}/{|\xi|}+bDu\right|\right) F(D^2u)=\tilde{f}(y) & \text{in $B_1 \cap \{y_n > \phi(y')\}$},\\
	u(y)=g(y) & \text{on $B_1 \cap \{y_n = \phi(y')\}$},
	\end{array} \right.
	\end{align*}
	where $\tilde{\Phi}(y, t):=\frac{\Phi(y, |\xi|t)}{\Phi(y, |\xi|)}, \tilde{f}(y):=\frac{f(y)}{\Phi(y, |\xi|)}$ and $b:=1/|\xi|$. Thus, if we choose $p_{0}>\max\{1, 6/\delta\}$, then we can apply \Cref{distlem2} for $u$.
	
	\item We next follow the contradiction argument of \Cref{basicreg} and the difference occurs when we employ the definition of limiting superjet and subjet:
	\begin{align*}
	 	\Phi(\hat{x}, |\xi+q_{x}|) F(X) \geq -\|f\|_{\infty},\\
		\Phi(\hat{y}, |\xi+q_{y}|) F(-Y) \leq\|f\|_{\infty}.
	\end{align*}
	This is due to the difference between equations \eqref{locDP} and \eqref{modeq2}, but we are still able to derive a contradiction. Recalling that $|q_{x}|, |q_{y}| \leq 2M\omega'(|\hat{x}-\hat{y}|)\leq 2M$, if we choose $p_{0} >3M$, then we have 
	\begin{align*}
	|\xi+q_{x}|, |\xi+q_{y}| \geq M.
	\end{align*}
	Combining this estimate with
	\begin{align*}
	-\|f\|_{\infty} \left(|\xi+q_x|^{-i(\Phi)}+|\xi+q_y|^{-i(\Phi)}\right) \leq -3\lambda M\omega_0|\hat{x}-\hat{y}|^{-1/2}+6[\Lambda(n-1)+\lambda]L,
	\end{align*}
	we conclude that 
	\begin{align*}
	3\lambda M\omega_0 \leq 2\|f\|_{\infty} M^{-i(\Phi)}+6[\Lambda(n-1)+\lambda]L,
	\end{align*}
	which is a contradiction. In this step, we have exploited the condition $i(\Phi) \geq 0$.
	\end{enumerate}

\end{proof}

\section{Global $C^{1, \alpha}$-regularity}\label{C1alpha}
We start this section with several reductions of the proof of  \Cref{thm:main}. First of all, by recalling \Cref{interior} which provides the interior $C^{1, \alpha}$-estimate of viscosity solutions, it is enough to show the local $C^{1, \alpha}$-estimate up to the boundary. Next, by following the proof of \cite[Theorem 1.1]{BBLL1}, we shall consider the degenerate case ($i(\Phi) \geq 0$) first, and then utilize this result for the singular case ($-1<i(\Phi)<0$). 

\begin{lemma}[$C^{1, \alpha}$-regularity up to the boundary; degenerate case]\label{keylemma}
	Suppose the assumptions \ref{a1}-\ref{a4} are in force with $i(\Phi) \geq 0$, $\phi(0)=0, \nabla \phi(0)=0$. Let $ \alpha$ be chosen to satisfy \eqref{alpha}.
	Then there exist constants $ \varepsilon_{0} \in (0,1)$, $\rho \in (0, 1/2)$ and $C_{0}>0$ depending on $\alpha$, $n$, $\lambda$, $\Lambda$, $\|D^{2}\phi\|_{L^{\infty}(\Omega)}$, $\|g\|_{C^{1, \beta_{g}}(\partial \Omega)}$, $i(\Phi)$ and $s(\Phi)$ such that for any $\xi \in \mathbb{R}^{n}$ and a viscosity solution, $u$, of
	\begin{align*}
	\left\{ \begin{array}{ll} 
	\Phi(y, |Du|)F(D^{2}u)=f(y) & \text{in $B_{1}(x) \cap \{y_n>\phi(y')\}$},\\
	u(y)=g(y) & \text{on $B_1(x) \cap \{y_n=\phi(y')\}$},
	\end{array} \right.
	\end{align*}
the following holds: if
\begin{align*}
\|u\|_{L^{\infty}(B_1(x) \cap \{y_n > \phi(y')\})} \leq 1 \quad \text{and} \quad \|f\|_{L^{\infty}(B_1(x) \cap \{y_n > \phi(y')\})} \leq \varepsilon_0,
\end{align*}
then there exists an affine function $l(y)=a+b \cdot(y-x)$ with $|a|+|b| \leq C_{0}$  such that for each $0 < r \leq \rho$, 
\begin{align*}
\|u-l\|_{L^{\infty}(B_r(x) \cap \{y_n > \phi(y')\})} \leq Cr^{1+\alpha}
\end{align*}
for some universal constant $C>0$.
\end{lemma}

Before we prove \Cref{keylemma} by using the induction, we first show the approximation lemma.
\begin{lemma}[Approximation lemma; degenerate case] \label{approxlem}
	Suppose \ref{a1}-\ref{a4} hold true with $i(\Phi) \geq 0$ and $\nu_0=\nu_1=1$. Let $\xi \in \mathbb{R}^n$ be an arbitrary vector and $u \in C(B_1(x) \cap \{y_n > \phi(y')\})$ be a viscosity solution of 
	\begin{align}\label{transeq}  
	\left\{ \begin{array}{ll} 
	\Phi(y, |\xi+Du|)F(D^{2}u)=f(y) & \text{in $B_{1}(x) \cap \{y_n>\phi(y')\}$},\\
	u(y)=g(y) & \text{on $B_1(x) \cap \{y_n=\phi(y')\}$},
	\end{array} \right.
	\end{align}
	 satisfying $\|u\|_{L^{\infty}(B_1(x) \cap \{y_n > \phi(y')\})} \leq 1$ and $\|g\|_{C^{1, \beta_g}(B_1(x) \cap \{y_n = \phi(y')\})} \leq 1$. Then for any $\mu>0$, there exists a constant $\varepsilon_0=\varepsilon_0(n, \lambda, \Lambda, i(\Phi), L, \mu)>0$ such that if
	\begin{align*}
	\|f\|_{L^{\infty}(B_1(x) \cap \{y_n > \phi(y')\})} \leq \varepsilon_0,
	\end{align*}
	then one can find a viscosity solution $h$ of an uniformly $(\lambda,\Lambda)$-elliptic equation
	\begin{align} \label{homoeq}
	\begin{cases}
	\mathcal{F}(D^2h)=0 \ &\text{in $B_{3/4}(x) \cap \{y_n > \phi(y')\}$},\\
	h=u \ & \text{on $B_{3/4}(x) \cap \{y_n=\phi(y')\}$}
	\end{cases}
	\end{align}
	such that
	\begin{align*}
	\|u-h\|_{L^{\infty}(B_{1/2}(x) \cap \{y_n > \phi(y')\})} \leq \mu.
	\end{align*}
\end{lemma}

\begin{proof}
	By contradiction, we suppose the conclusion of the lemma fails. Therefore, there exist $\mu_0>0$ and sequences of $\{F_k\}_{k=1}^{\infty}$, $\{\Phi_k\}_{k=1}^{\infty}$, $\{f_k\}_{k=1}^{\infty}$, $\{g_k\}_{k=1}^{\infty}$, and $\{u_k\}_{k=1}^{\infty}$ and a sequence of vectors $\{\xi_k\}_{k=1}^{\infty}$ such that
	\begin{description}
		\item[(C1) \label{c1}] $F_k \in C(\mathcal{S}(n), \mathbb{R})$ is uniformly $(\lambda, \Lambda)$-elliptic;
		\item[(C2) \label{c2}] for $\Phi_k \in C(B_1 \times [0, \infty), [0, \infty))$, the map $\textstyle t \mapsto \frac{\Phi_k(x,t)}{t^{i(\Phi)}}$ is almost non-decreasing and the map $\textstyle t \mapsto \frac{\Phi_k(x,t)}{t^{s(\Phi)}}$ is almost non-increasing with constant $L \geq 1$, and $\Phi_k(y, 1)=1$ for all $y \in B_1(x) \cap \{y_n>\phi(y')\}$; 
		\item[(C3) \label{c3}] $f_k \in C(B_1(x) \cap \{y_n>\phi(y')\})$ with $\|f_k\|_{L^{\infty}(B_1(x) \cap \{y_n>\phi(y')\})} \leq 1/k$;
		\item[(C4) \label{c4}] $u_k \in C(B_1(x) \cap \{y_n>\phi(y')\})$ with $\|u\|_{L^{\infty}(B_1(x) \cap \{y_n>\phi(y')\})} \leq 1$ solves the equation
		\begin{align*}  
		\left\{ \begin{array}{ll} 
		\Phi_{k}(y, |\xi_k+Du_k|) F_k(D^2u_k)=f_k(y) & \text{in $B_1(x) \cap \{y_n > \phi(y')\}$},\\
		u_k(y)=g_k(y) & \text{on $B_1(x) \cap \{y_n=\phi(y')\}$},
		\end{array} \right.
		\end{align*}
		with $\|g_k\|_{C^{1, \beta_g}(B_1(x) \cap \{y_n=\phi(y')\})} \leq 1$, but
		\begin{align}\label{contr}
			\|u_k-h\|_{L^{\infty}(B_{1/2}(x) \cap \{y_n > \phi(y')\})} > \mu_0 \quad \text{for any $k \in \mathbb{N}$},
		\end{align}
		for any $h$ satisfying \eqref{homoeq}.
	\end{description}
The condition \ref{c1} implies that $F_k$ converges to some uniformly $(\lambda, \Lambda)$-elliptic operator $F_{\infty} \in C(\mathcal{S}(n), \mathbb{R})$. Similarly, the condition \ref{c4} implies that $g_k$ converges to $g_{\infty}$ uniformly. For a further discussion, we consider two cases:
\begin{enumerate}[(i)]
	\item ($\{\xi_k\}_{k=1}^{\infty}$ is bounded) Up to a subsequence, $\xi_k$ converges to some vector $\xi_{\infty}$. Then we consider a sequence $\{\tilde{u}_k\}_{k=1}^{\infty}:=\{u_k+x \cdot \xi_k\}_{k=1}^{\infty}$ satisfying
	\begin{align*}
	\left\{ \begin{array}{ll} 
	\Phi_{k}(y, |D\tilde{u}_k|) F(D^2\tilde{u}_k)=f_k(y) & \text{in $B_1(x) \cap \{y_n > \phi(y')\}$},\\
	\tilde{u}_k(y)=\tilde{g}_k(y) & \text{on $B_1(x) \cap \{y_n=\phi(y')\}$}
	\end{array} \right.
	\end{align*}
	for $\tilde{g}_k(x):=g_k(x)+x\cdot \xi_k$. Therefore, we can apply \Cref{basicreg} for $\tilde{u}_k$ and so by Arzela-Ascoli theorem, we conclude that $u_k \to u_{\infty}$ uniformly in $B_r(x) \cap \{y_n >\phi(y')\}$ for any $0<r<1$. Indeed, $u_{\infty}$ satisfies
	\begin{align*}
	\left\{ \begin{array}{ll} 
	\Phi_{\infty}(y, |\xi_{\infty}+Du_{\infty}|) F(D^2u_{\infty})=0 & \text{in $B_1(x) \cap \{y_n > \phi(y')\}$},\\
	u_{\infty}(y)=g_{\infty}(y) & \text{on $B_1(x) \cap \{y_n=\phi(y')\}$}.
	\end{array} \right.
	\end{align*}
	Then by following the proof of \cite[Lemma 4.1]{BBLL1}, we conclude that 
	\begin{align*}
		F_{\infty}(D^2u_{\infty})=0 \quad \text{in $B_{3/4}(x) \cap \{y_n >\phi(y')\}$},
	\end{align*}
	which leads to the contradiction with \eqref{contr} (choose $h=u_{\infty}$ and $\mathcal{F}=F_{\infty}$).
	
	\item ($\{\xi_k\}_{k=1}^{\infty}$ is unbounded) In this case, for the constant $p_0>0$ chosen in \Cref{basicreg2}, we may assume $|\xi_k|>p_0$ and $|\xi_k| \to \infty$ (up to a subsequence). Thus, we can apply \Cref{basicreg2} for $u_k$ and so by Arzela-Ascoli theorem, we conclude that $u_k \to u_{\infty}$ uniformly in $B_r(x) \cap \{y_n >\phi(y')\}$ for any $0<r<1$. Again by following the proof of \cite[Lemma 4.1]{BBLL1}, we conclude that 
		\begin{align*}
		F_{\infty}(D^2u_{\infty})=0 \quad \text{in $B_{3/4}(x) \cap \{y_n >\phi(y')\}$},
		\end{align*}
		which leads to the contradiction with \eqref{contr} (choose $h=u_{\infty}$ and $\mathcal{F}=F_{\infty}$).
\end{enumerate} 

\end{proof}

\begin{proof}[Proof of \Cref{keylemma}]
	By the smallness regime in \Cref{rmk:small}, we may assume that $u \in C(\Omega)$ is a viscosity solution with
	\begin{align*}
	\|u\|_{L^{\infty}(B_{1} \cap \{y_n > \phi(y')\})} \leq 1, \, \|g\|_{C^{1, \beta_g}(B_1 \cap \{y_n = \phi(y')\})} \leq 1,  \, \|f\|_{L^{\infty}(B_{1} \cap \{y_n > \phi(y')\})} \leq \varepsilon_{0},
	\end{align*}
	and $\nu_0=\nu_1=1$. As in \cite{BBLL1, BPRT1, dSJRR1}, the proof is based on the induction argument: we claim that there exist universal constants $0<\rho \ll 1$, $C_{0}>1$ and a sequence of affine functions
	\begin{align*}
	 	l_{k}(y):=a_{k}+b_{k} \cdot (y-x),
	\end{align*}
where $\{a_{k}\}_{k=1}^{\infty} \subset \mathbb{R}$ and $\{b_{k}\}_{k=1}^{\infty} \subset \mathbb{R}^{n}$ such that, for every $k\in \mathbb{N}$,
\begin{description}
	\item[(E1)\label{e1}] $\sup_{y \in B_{\rho^{k}}(x) \cap \{y_n > \phi(y')\}}|u(y)-l_{k}(y)| \leq \rho^{k(1+\alpha)}$;
	
	\item[(E2)\label{e2}] $|a_{k}-a_{k-1}| \leq C_{0} \rho^{(k-1)(1+\alpha)}$ and $|b_{k}-b_{k-1}|\leq C_{0}\rho^{(k-1)\alpha}$.
\end{description}	
\begin{enumerate}[(i)]
\item (Initial step) Without loss of generality, we may assume $x=0$.  Let us set 
\begin{align*}
 	l_{1}(y):=h(0)+Dh(0) \cdot y,
\end{align*}
where $h$ is the approximation function coming from \Cref{approxlem} for a constant $\mu>0$ to be determined later. Then, by the boundary estimate for uniformly elliptic fully nonlinear operators \cite{AS1, LZ1}, there exists a constant $C_{0} =C_{0}(n, \lambda, \Lambda)>1$ such that
\begin{align*}
 	\|h\|_{C^{1, \alpha}(B_{1/2} \cap \{y_n > \phi(y')\})} \leq C_{0} \quad \text{and} \quad \sup_{y \in B_{\rho} \cap \{y_n > \phi(y')\}} |h(y)-l_{1}(y)| \leq C_{0}\rho^{1+\beta_{g}},
\end{align*}
for every $0<\rho \leq 1/2$. The triangle inequality yields that
\begin{align*}
 	\sup_{y \in B_{\rho} \cap \{y_n > \phi(y')\}} |u(y)-l_{1}(y)| \leq C_{0}\rho^{1+\beta_{g}}+\mu.
\end{align*}
We now select a universal constant $0<\rho \ll 1$ small enough so that
\begin{align*}
 	C_{0}\rho^{1+\beta_{g}} \leq \frac{1}{2}\rho^{1+\alpha}, \quad \rho^{\alpha} \leq \frac{1}{2}, \quad \text{and} \quad \rho^{1-\alpha(1+s(\Phi))} \leq 1,
\end{align*}
which is possible due to the choice of $\alpha$; recall \eqref{alpha}. In a sequel, we choose a constant $\mu:=\rho^{1+\alpha}/2$ and set $a_{0}=0, b_{0}=0$, $a_{1}=h(0)$ and $b_{1}=Dh(0)$, which completes the proof of the initial step.

\item (Iterative procedure) We now suppose that \ref{e1} and \ref{e2} hold true for $k \geq 1$. We then verify \ref{e1} and \ref{e2} for $k+1$. For this purpose, we define a rescaled function
\begin{align*}
 	u_{k}(y):=\frac{u(\rho^{k}y)-l_{k}(\rho^{k}y)}{\rho^{k(1+\alpha)}}.
\end{align*}
Then $u_{k}$ satisfies 
\begin{align*}  
	\left\{ \begin{array}{ll} 
	\Phi_{k}(y, |\xi_{k}+Du_{k}|)F(D^{2}u_{k})=f_{k}(y) & \text{in $B_{1} \cap \{y_n>\phi_{k}(y')\}$},\\
	u_{k}(y)=g_{k}(y) & \text{on $B_1 \cap \{y_n=\phi_{k}(y')\}$},
	\end{array} \right.
	\end{align*}
where
\begin{align*}
	F_{k}(M)&:=\rho^{k(1-\alpha)}F(\rho^{k(\alpha-1)}M),\quad
	\Phi_{k}(y,t):=\frac{\Phi(\rho^{k}y, \rho^{k\alpha}t)}{\Phi(\rho^{k}y, \rho^{k\alpha})},\quad
	f_{k}(y):=\frac{\rho^{k(1-\alpha)}}{\Phi(\rho^{k}y, \rho^{k\alpha})}f(\rho^{k}y),\\
	g_{k}(y)&:=\frac{g(\rho^{k}y)-l_{k}(\rho^{k}y)}{\rho^{k(1+\alpha)}},\quad
	\phi_{k}(y'):=\rho^{-k}\phi(\rho^{k}y') \text{ and }
	\xi_{k}:=\rho^{-k\alpha}b_{k}.
\end{align*}
It can be easily checked that 
\begin{enumerate}
\item $F_{k}$ satisfies \ref{a1} with the same constants $(\lambda, \Lambda)$;

\item $\Phi_{k}$ satisfies \ref{a2} with the same constants $(i(\Phi), s(\Phi))$ and $\Phi_{k}(y,1) \equiv 1$;

\item $\|u_{k}\|_{L^{\infty} (B_{1} \cap \{y_{n} > \phi_{k}(y') \} ) } \leq 1$ by the induction hypothesis;

\item $\|f_{k}\|_{L^{\infty} (B_{1}\cap \{y_{n} > \phi_{k}(y') \})} \leq L \varepsilon_{0} \rho^{k(1-\alpha(1+s(\Phi)))} \leq L \varepsilon_{0}$;

\item $\|D^{2}\phi_{k}\|_{L^{\infty} (B_{1}\cap \{y_{n} > \phi_{k}(y') \})}  \leq \rho^{k}\|D^{2}\phi\|_{L^{\infty} (B_{1}\cap \{y_{n} > \phi(y') \})} \leq \|D^{2}\phi\|_{L^{\infty} (B_{1}\cap \{y_{n} > \phi(y') \})}$.
\end{enumerate}
Moreover, for $g_{k}$, we can compute
\begin{align*}
|Dg_{k}(y)-Dg_{k}(z)|=\rho^{-k \alpha} |Dg(\rho^{k}y)-Dg(\rho^{k}z)| &\leq \|g\|_{C^{1, \beta_{g}}(\partial \Omega) } \cdot \rho^{-k \alpha} |\rho^{k}(y-z)|^{\beta_{g}} \\
&\leq \|g\|_{C^{1, \beta_{g}}(\partial \Omega) } \cdot |y-z|^{\beta_{g}}.
\end{align*}
By recalling the fact that $u=g$ on $\partial \Omega$ together with the induction hypothesis, we can conclude that
\begin{align*}
\|g_{k}\|_{C^{1, \beta_{g}}(B_{1}\cap \{y_{n} > \phi_{k}(y') \} )} \leq \|g\|_{C^{1, \beta_{g}}(\partial \Omega) } \leq 1.
\end{align*}
Hence, we now apply \Cref{approxlem} for $u_{k}$ and then follow the argument in the initial step to ensure the existence of an affince function $\overline{l}(y):=\overline{a}+\overline{b} \cdot y$ such that
\begin{align*}
\sup_{y \in B_{\rho} \cap \{y_n > \phi_{k}(y')\}}|u_{k}(y)-\overline{l}(y)| \leq \rho^{1+\alpha} \quad \text{and} \quad |\overline{a}|, |\overline{b}| \leq C_{0}.
\end{align*}
By scaling back, we conclude that
\begin{align*}
\sup_{y \in B_{\rho^{k+1}} \cap \{y_n > \phi(y')\}}|u(y)-l_{k+1}(y)| \leq \rho^{(k+1)(1+\alpha)},
\end{align*}
where
\begin{align*}
l_{k+1}(y):=l_{k}(y)+\rho^{k(1+\alpha)}\cdot\overline{l}(\rho^{-k}y).
\end{align*}
Here note that
\begin{align*}
|a_{k+1}-a_{k}|&=\rho^{k(1+\alpha)}|\overline{a}| \leq C_{0}\rho^{k(1+ \alpha)},\\
|b_{k+1}-b_{k}|&=\rho^{k\alpha}|\overline{b}| \leq C_{0}\rho^{k\alpha}.
\end{align*}
Therefore, \ref{e1} and \ref{e2} hold for $k+1$.
\end{enumerate}
\end{proof}

\begin{proof}[Proof of \Cref{thm:main}]
	We first consider the degenerate case, i.e., $i(\Phi) \geq 0$. By \Cref{keylemma}, a viscosity solution can be approximated by an affine function with an error of order $r^{1+\alpha}$. By following the argument in \cite{BPRT1}, we can derive the desired global $C^{1, \alpha}$-estimate provided that $i(\Phi) \geq 0$.
	
	On the other hand, for the singular case ($i(\Phi) <0$), we employ the idea of \cite{BBLL1}. Indeed, we claim that \Cref{keylemma} still holds for the singular case. \Cref{basicreg} guarantees that 
	\begin{align*}
	 	\|u\|_{C^{0,1}(B_{3/4}\cap \{y_{n}>\phi(y')\})} \leq c,
	\end{align*}
	for a universal constant $c>0$. Then $u$ is a viscosity solution of 
	\begin{align*}  
	\left\{ \begin{array}{ll} 
	\tilde{\Phi}(y, |Du|)F(D^{2}u)=\tilde{f}(y) & \text{in $B_{3/4} \cap \{y_n>\phi(y')\}$},\\
	u(y)=g(y) & \text{on $B_{3/4} \cap \{y_n=\phi(y')\}$},
	\end{array} \right.
	\end{align*}
where
\begin{align*}
	\tilde{\Phi}(y,t):=t^{-i(\Phi)}\Phi(y, t) \quad \text{and} \quad
	\tilde{f}(y):=|Du|^{-i(\Phi)}f(y).
\end{align*}
Here $\tilde{\Phi}$ satisfies the condition \ref{a2} with $i(\tilde{\Phi})=0$, $s(\tilde{\Phi})=s(\Phi)-i(\Phi)$, and 
\begin{align*}
\|\tilde{f}\|_{L^{\infty}(\Omega)} \leq c^{-i(\Phi)} \varepsilon_{0}.
\end{align*}
Thus, one can repeat the argument in the proof of \Cref{keylemma} to obtain the global $C^{1, \alpha}$-estimate.
\end{proof}

\section{Comparison principle and Perron's method}\label{Perron}
The purpose of this section is to study the classical result of comparison principle and as a consequence, to deduce the existence of a viscosity solution to \eqref{me} by Perron's method. Nevertheless, the assumptions \ref{a1}-\ref{a4} are not sufficient to obtain the aforementioned results. Therefore, we require an additional assumption \ref{a5} which guarantees the comparison principle for approximated Dirichlet problems. Before we precisely state this new assumption, we summarize known results regarding the comparison principle.

\begin{remark}[Comparison principle]
	Let $H : \mathbb{R}^{n} \times \mathbb{R} \times \mathbb{R}^{n} \times \mathcal{S}^{n} \to \mathbb{R}$ be a \textit{proper} map. In other words, $H$ satisfies
	\begin{align}\label{proper}
	H(x, r, p, X) &\leq H(x,r,p,Y) \quad \text{whenever $X \leq Y$},\\
	H(x, r, p, X) &\leq H(x,s,p,X) \quad \text{whenever $s \leq r$.}
	\end{align} 
	Then we say that $H$ satisfies the \textit{comparison principle} if the following holds:

Let $v \in \mathrm{USC}(\overline{\Omega})$ [resp. $w \in \mathrm{LSC}(\overline{\Omega})$] be a subsolution [resp. supersolution] of $H=0$ in $\Omega$ and $v \leq w$ on $\partial \Omega$. Then $v \leq w$ in $\overline{\Omega}$.

We refer to \cite{BD1, CIL1, GT1, KK1, KK2} for several sufficient conditions of then comparison principle. In short, $H$ satisfies the comparison principle if  $H(x, r, p, X)$ is independent of $x$, and one of the following conditions holds:
\begin{enumerate}[(i)]
\item $H(x, r, p, X)$ is strictly decreasing in $r$ and $H$ is degenerate elliptic (i.e., $H$ satisfies \eqref{proper}), or

\item $H(x, r, p, X)$ is nonincreasing in $r$ and $H$ is uniformly elliptic.
\end{enumerate}
It is noteworthy that the condition that $H$ is independent of $x$ can be relaxed to some extra structure conditions on $H$, which display a kind of smoothness on $H$ with respect to $x$-variable; see \cite{CIL1, GT1, KK1} for details.
\end{remark}

We now consider a proper map 
\begin{align*}
H(x, u, Du, D^{2}u):=\Phi(x, |Du|)F(D^{2}u)-f(x)
\end{align*}
with \ref{a1}-\ref{a4}. 
It is easily checked that $H$ is degenerate elliptic, nonincreasing in $r$, and $H$ depends on $x$. In view of the previous remark, we cannot expect the comparison principle for $H$ and Perron's method for the associated Dirichlet problem. To overcome this challenge, we will impose an additional structure condition on $H$ and approximate the map $H$ to ensure the strict monotonicity with respect to $r$-variable:
\begin{enumerate}[(i)]
\item 
\begin{description}
	\item[(A5)\label{a5}] There exists a continuous function $\omega : [0, \infty) \to [0, \infty)$ such that $\omega(0)=0$ and 
	\begin{align*}
	\Phi(x, \alpha |x-y|)F(X)-\Phi(y, \alpha |x-y|)F(-Y) \leq \omega(\alpha|x-y|^{2}+|x-y|),
	\end{align*}
	whenever $\alpha>0$, $x, y \in \Omega$, $X, Y \in \mathcal{S}^{n}$ and
	\begin{align}
		- 3\alpha
		\begin{pmatrix}
			I & 0 \\
			0 & I
		\end{pmatrix}
		\leq
		\begin{pmatrix}
			X & 0 \\
			0 & Y
		\end{pmatrix}
		\leq
		3\alpha
		\begin{pmatrix}
			I & -I \\
			-I & I
		\end{pmatrix}.
	\end{align} 
\end{description}
Similar conditions with \ref{a5} can be found in \cite{BD1, CIL1}. 

\item Let us consider the approximated problem given by
	\begin{align}
		\label{AE:1}
		H_{\varepsilon}(x, u, Du, D^{2}u):= \Phi(x,|Du|)F(D^2u)-f(x) - \varepsilon u= 0\quad\text{in}\quad \Omega
	\end{align}
	for $\varepsilon>0$. Clearly, $H_{ \varepsilon}$ is strictly decreasing in $r$.
\end{enumerate}

\begin{lemma}[Comparison principle II]
	\label{thm_CP}
	Suppose that the assumptions \ref{a1}, \ref{a2}, \ref{a5} are in force and $f\in C(\overline{\Omega})$. Then $H_{ \varepsilon}$ satisfies the comparison principle: 
	
Let $v$ and $w$ be a viscosity subsolution and a supersolution of \eqref{AE:1}, respectively. If $v\leq w$ on $\partial\Omega$, then $v\leq w$ in $\Omega$.
\end{lemma}
\begin{proof}
	By contradiction, we suppose that 
	\begin{align*}
	L_0:= \sup\limits_{x\in \overline{\Omega}}(v(x)-w(x))>0.
	\end{align*}
	For any $\alpha>0$, we define 
	\begin{align*}
		L_{\alpha}:= \sup\limits_{x,y\in\overline{\Omega}}\left[v(x)-w(y)-(\alpha/2)|x-y|^{2} \right]
	\end{align*}
	and clearly $L_{\alpha}\geq L_0$. Suppose that the maximum $L_{\alpha}$ is attained at a point $(x_{\alpha},y_{\alpha})\in \overline{\Omega}\times \overline{\Omega}$. It implies from \cite[Lemma 3.1]{CIL1} that 
	\begin{align*}
		\lim\limits_{\alpha\to \infty } \alpha |x_{\alpha}-y_{\alpha}|^{2} = 0.
	\end{align*}
	This one and the fact that $v \leq w$ on $\partial\Omega$ yield that $x_{\alpha},y_{\alpha}\in \Omega$ for $\alpha>0$ large enough. At this moment, we are able to apply Ishii-Lions lemma, \cite[Theorem 3.2]{CIL1}, to ensure that there exist a limiting super-jet $\left( \alpha(x_{\alpha}-y_{\alpha}), X_{\alpha} \right)$ of $v$ at $x_{\alpha}$ and a limiting sub-jet $\left( \alpha(x_{\alpha}-y_{\alpha}), -Y_{\alpha} \right)$ of $w$ at $y_{\alpha}$ so that 
	\begin{align*}	
		- 3\alpha
		\begin{pmatrix}
			I & 0 \\
			0 & I
		\end{pmatrix}
		\leq
		\begin{pmatrix}
			X_{\alpha} & 0 \\
			0 & Y_{\alpha}
		\end{pmatrix}
		\leq
		3\alpha
		\begin{pmatrix}
			I & -I \\
			-I & I
		\end{pmatrix}
	\end{align*}
	and 
	\begin{align*}
		\begin{cases}
			\Phi
\left(x_{\alpha}, \alpha|x_{\alpha}-y_{\alpha}| \right)
F(X_{\alpha}) - \varepsilon v(x_{\alpha}) 
			\geq f(x_{\alpha}), \\
			\Phi
\left(y_{\alpha}, \alpha|x_{\alpha}-y_{\alpha}| \right)
F(-Y_{\alpha}) - \varepsilon w(y_{\alpha}) 
			\leq f(y_{\alpha}).
		\end{cases}
	\end{align*}
By using the relation $L_{0}=\lim_{\alpha \to \infty}[v(x_{\alpha})-w(y_{\alpha})]$ and the assumption \ref{a5}, we have, for sufficiently large $\alpha>0$,
\begin{align*}
 	\frac{ \varepsilon L_{0}}{2} &\leq \varepsilon [v(x_{\alpha})-w(y_{\alpha})]\\ &\leq f(y_{\alpha})-f(x_{\alpha})+\Phi\left(x_{\alpha}, \alpha|x_{\alpha}-y_{\alpha}| \right)
F(X_{\alpha})-\Phi\left(y_{\alpha}, \alpha|x_{\alpha}-y_{\alpha}| \right)
F(-Y_{\alpha})\\
&\leq f(y_{\alpha})-f(x_{\alpha})+\omega(\alpha|x_{\alpha}-y_{\alpha}|^{2}+|x_{\alpha}-y_{\alpha}|).
\end{align*}
Since $f\in C(\overline{\Omega})$ and $\omega(0+)=0$, we arrive at a contradiction when $\alpha \to \infty$.
\end{proof}


We now turn our attention to showing the existence of a viscosity solution to \eqref{AE:1}.

\begin{lemma}[Existence of sub/supersolutions]
	\label{thm_exist}
	Suppose the assumptions \ref{a1}-\ref{a4} are in force. Then for every $\varepsilon\in (0,1)$, there exist a viscosity subsolution $v_{\varepsilon}\in C(\overline{\Omega})$ and a viscosity supersolution $w_{\varepsilon}\in C(\overline{\Omega})$ of \eqref{AE:1} with $v_{\varepsilon}= w_{\varepsilon} = g$ on $\partial\Omega$.  Moreover, there exists a positive constant $c \equiv c(n, \lambda, \Lambda, \nu_0, L, r, \mathrm{diam}(\Omega), \|f\|_{L^{\infty}(\Omega)}, \|g\|_{L^{\infty}(\partial \Omega)})$ such that
	\begin{align*}
	-c \leq v_{ \varepsilon} \leq w_{ \varepsilon} \leq c, \quad \text{for any $0< \varepsilon<1$.}
	\end{align*}
\end{lemma}

\begin{proof}
	 Let $z\in\partial\Omega$ be a fixed point. There exists a point $x_{z}\in\R^{n}\setminus\overline{\Omega}$ such that $\overline{B_{r}(x_{z})}\cap \overline{\Omega} = \{z\}$ with $r = |z-x_{z}|$ since $\Omega$ satisfies a uniform exterior sphere property; see \Cref{ballcond}. We now consider a function $v_{z} : \overline{\Omega}\rightarrow [0,\infty)$ defined by 
	\begin{align*}
		v_{z}(x) := K \left( r^{-\kappa_{0}} - |x-x_{z}|^{-\kappa_{0}} \right) 
	\end{align*}
	for positive constants $\kappa_{0}:=\frac{n\Lambda+1}{\lambda}$ and $K \geq \min\left\{1, \frac{R^{\kappa_0+1}}{\kappa_0}\right\}$ to be determined later, with $R:=r+\mathrm{diam}(\Omega)$. Note that $v_{z}(z)=0$, $v_{z}>0$ in $\Omega$, and direct calculations yield that
	\begin{align*}
		Dv_{z}(x) = K\kappa_{0}\frac{x-x_{z}}{|x-x_{z}|^{\kappa_{0}+2}}
	\end{align*}
	and 
	\begin{align*}
		D^{2}v_{z}(x) = K\kappa_{0}\frac{I}{|x-x_{z}|^{\kappa_{0}+2}} - K\kappa_{0}(\kappa_{0}+2)\frac{(x-x_{z})\otimes (x-x_{z})}{|x-x_{z}|^{\kappa_{0}+4}}.
	\end{align*}
	Due to the choice of $\kappa_0$, we have
	\begin{align*}
		F(D^2v_z(x)) \leq  \frac{K\kappa_0}{|x-x_z|^{\kappa_0+2}}\left((n-1)\Lambda-(\kappa_0+1)\lambda\right) \leq -\frac{K\kappa_0}{|x-x_z|^{\kappa_0+2}}.
	\end{align*}
	On the other hand, for a fixed $\delta \in (0,1)$, we further define 
	\begin{align*}
		v_{z, \delta}(x):=g(z)+\delta+M_{\delta}v_{z}(x),
	\end{align*}
	where the constant $M_{\delta} \geq 1$ can be chosen so that $v_{z, \delta} \geq g$ on $\partial \Omega$. This is possible because $K \geq 1$ and $g$ is continuous on $\partial \Omega$. Indeed, $M_{\delta}$ depends only on the modulus of continuity of $g$, and is independent of $z$. Then we see that 
	\begin{align*}
		&\Phi(x, |Dv_{z,\delta}(x)|)F(D^{2}v_{z,\delta}(x))- \varepsilon v_{z,\delta}(x) \notag\\&
		\leq
		\frac{\nu_0}{L}\min\left\{ \left( \frac{K\kappa_{0}}{|x-x_{z}|^{\kappa_0 +1}} \right)^{i(\Phi)}, \left(\frac{K\kappa_{0}}{|x-x_{z}|^{\kappa_0 +1}} \right)^{s(\Phi)} \right\} \cdot \left(-M_{\delta}\frac{K\kappa_{0}}{ |x-x_{z}|^{\kappa_0 + 2}}\right)+\varepsilon \norm{g}_{L^{\infty}(\partial\Omega)}.
	\end{align*}
Here if we let $A:=\frac{K\kappa_{0}}{|x-x_{z}|^{\kappa_0 +1}}$, then 
	\begin{align*}
		\min\left\{A^{i(\Phi)}, A^{s(\Phi)} \right\}&=A^{-1} \min\left\{A^{i(\Phi)+1}, A^{s(\Phi)+1} \right\}\\
		&\geq \left(\frac{K\kappa_0}{r^{\kappa_0+1}}\right)^{-1} \min\left\{\left(\frac{K\kappa_0}{R^{\kappa_0+1}}\right)^{i(\Phi)+1}, \left(\frac{K\kappa_0}{R^{\kappa_0+1}}\right)^{s(\Phi)+1} \right\}.
	\end{align*}
	Note that we need to select $K$ so that $\frac{K\kappa_0}{R^{\kappa_0+1}} \geq 1$. Since $0 < i(\Phi)+1 \leq s(\Phi)+1$, we conclude that
	\begin{align*}
		\min\left\{A^{i(\Phi)}, A^{s(\Phi)} \right\} \geq \left(\frac{\kappa_0}{r^{\kappa_0+1}}\right)^{-1} \left(\frac{\kappa_0}{R^{\kappa_0+1}}\right)^{i(\Phi)+1}K^{i(\Phi)}.
	\end{align*}
Hence, we deduce that
\begin{align*}
	\Phi(x, |Dv_{z,\delta}(x)|)&F(D^{2}v_{z,\delta}(x))- \varepsilon v_{z,\delta}(x)\\
	&\leq -\frac{\nu_0}{L} \left(\frac{\kappa_0}{r^{\kappa_0+1}}\right)^{-1} \left(\frac{\kappa_0}{R^{\kappa_0+1}}\right)^{i(\Phi)+1}\frac{\kappa_0}{R^{\kappa_0+2}}K^{i(\Phi)+1}+ \norm{g}_{L^{\infty}(\partial\Omega)}. 
\end{align*}
Therefore, we can choose $K=K(n, \lambda, \Lambda, \nu_0, L, r, \mathrm{diam}(\Omega), \|f\|_{L^{\infty}(\Omega)}, \|g\|_{L^{\infty}(\partial \Omega)})$ large enough so that
\begin{align*}
	\Phi(x, |Dv_{z,\delta}(x)|)&F(D^{2}v_{z,\delta}(x))- \varepsilon v_{z,\delta}(x) \leq -\|f\|_{L^{\infty}(\Omega)} \quad \text{in $\Omega$},
\end{align*}
i.e., $v_{z, \delta}$ is a viscosity supersolution to \eqref{AE:1}.

Finally, we define 
\begin{align*}
	w_{\varepsilon}(x):=\inf\left\{v_{z, \delta}(x) \text{: $z \in \partial \Omega$ and $\delta \in (0,1)$} \right\}.
\end{align*}
It is easy to check that $w_{\varepsilon}$ is a viscosity supersolution to \eqref{AE:1} in $\Omega$ and enjoys the boundary condition $w_{\varepsilon}=g$ on $\partial \Omega$. Moreover, it immediately follows from the construction of $w_{ \varepsilon}$ that 
\begin{align*}
w_{\varepsilon} \leq C(n, \lambda, \Lambda, \nu_0, L, \mathrm{diam}(\Omega), \|f\|_{L^{\infty}(\Omega)}, \|g\|_{L^{\infty}(\partial \Omega)}) \quad \text{in $\Omega$},
\end{align*}
for any $ \varepsilon \in (0,1)$. The existence of a viscosity subsolution $v_{\varepsilon}$ and its lower bound can be shown in a similar argument. Finally, since $v_{ \varepsilon}=w_{ \varepsilon}=g$ on $\partial \Omega$, \Cref{thm_CP} implies that $v_{ \varepsilon} \leq w_{ \varepsilon}$ in $\Omega$.
\end{proof}

\begin{proof}[Proof of \Cref{thm:mthm1}]
	An application of Perron's method \cite[Theorem 4.1]{CIL1} together with \Cref{thm_CP} and \Cref{thm_exist} yields the existence of a viscosity solution $u_{\varepsilon}$ to the approximated equation \eqref{AE:1} with the boundary condition $u_{\varepsilon}=g$.
	
We now understand $u_{ \varepsilon}$ as a viscosity solution of 
\begin{align*}
		\left\{\begin{array}{rclcc}
		\Phi(x,|Du_{ \varepsilon}|)F(D^2u_{ \varepsilon}) &=& f_{ \varepsilon}(x) & \mbox{in }& \O, \\
		u_{ \varepsilon}(x)&=&g(x) & \mbox{on }& \partial\O,
		\end{array}\right.	  
\end{align*}
where $f_{ \varepsilon}(x):=f(x)+ \varepsilon u_{ \varepsilon}(x)$. Here note that $\{f_{ \varepsilon}\}_{ \varepsilon \in (0,1)}$ is uniformly bounded in $L^{\infty}(\Omega)$ by \Cref{thm_exist}. Then by applying \cite[Lemma 3.1]{BBLL1} and \Cref{basicreg} (or just by applying the stronger result \Cref{thm:main}), we have that $\{u_{ \varepsilon}\}_{ \varepsilon \in (0,1)}$ is uniformly bounded in $C^{0, \gamma}(\overline{\Omega})$ for some $\gamma \in (0,1)$. Therefore, we can extract a uniformly converging subsequence such that $u_{ \varepsilon_{j}} \to u_{\infty}$ when $ \varepsilon_{j} \to 0$, and by the standard stability argument, we conclude that $u_{\infty}$ solves \eqref{me}.
\end{proof}

\vspace{0.2cm}

\noindent
{\bf Conflict of interest.}
The authors declare that they have no conflict of interest.

\vspace{0.2cm}

\noindent
{\bf Acknowledgments.} This work is supported by the National Research Foundation of Korea (NRF) grant funded by the Korea government (MSIP): NRF-2021R1A4A1027378. Se-Chan Lee is supported by Basic Science Research Program through the National Research Foundation of Korea (NRF) funded by the Ministry of Education (2022R1A6A3A01086546).

\bibliographystyle{amsplain}

\begin{thebibliography}{10}

\bibitem{AKSZ1} H. Aikawa, T. Kilpel\"{a}inen, N. Shanmugalingam, and X. Zhong, 
\textit{Boundary {H}arnack principle for {$p$}-harmonic functions in
		smooth {E}uclidean domains},
	Potential Anal. \textbf{26} (2007), no. 3, 281–301.


\bibitem{ART1} D. J. Ara\'ujo, G. Ricarte, and E. V. Teixeira, 
\textit{Geometric gradient estimates for solutions to degenerate elliptic equations},
Calc. Var. Partial Differential Equations \textbf{53} (2015), no. 3-4, 605-625.


\bibitem{AS1} D. Ara\'ujo and B. Sirakov,
 \textit{Sharp boundary and global regularity for degenerate fully nonlinear elliptic equations}, 
arXiv preprint arXiv:2108.01150 (2021).

\bibitem{BBLL1} S. Baasandorj, S.-S. Byun, K.-A. Lee, and S.-C. Lee,
\textit{${C}^{1, \alpha}$-regularity for a class of degenerate/singular fully nonlinear elliptic equations},
arXiv preprint arXiv:2209.14581 (2022).

\bibitem{BD1} I. Birindelli and F. Demengel, 
\textit{Comparison principle and Liouville type results for singular fully nonlinear operators},
Ann. Fac. Sci. Toulouse Math. (6) \textbf{13} (2004), no. 2, 261-287.


\bibitem{BD2} I. Birindelli and F. Demengel, 
\textit{Regularity and uniqueness of the first eigenfunction for singular fully nonlinear operators}, 
J. Differential Equations \textbf{249} (2010),  no. 5, 1089-1110.


 

\bibitem{BD4} I. Birindelli and F. Demengel, 
\textit{$C^{1,\beta}$ regularity for Dirichlet problems associated to fully nonlinear degenerate elliptic equations}, 
ESAIM Control Optim. Calc. Var. \textbf{20} (2014), no. 4, 1009-1024.
 

 

\bibitem{BPRT1} A. C. Bronzi, E. A. Pimentel, G. C. Rampasso, and E. V. Teixeira, 
\textit{Regularity of solutions to a class of variable-exponent fully nonlinear elliptic equations}, 
J. Func. Anal. \textbf{279} (2020), no. 12, Paper 108781.
 

\bibitem{CC1} L. Caffarelli and C. Cabr\'e,
\textit{Fully nonlinear equations},
American Mathematical Society Colloquium Publications 43 (American Mathematical Society, Providence, RI, 1995).


\bibitem{CIL1} M. G. Crandall, H. Ishii, and P. L. Lions, 
\textit{User's guide to viscosity solutions of second order partial differential equations},
Bull. Amer. Math. Soc. \textbf{27} (1992), 1-67.


\bibitem{SR1} J. V. Da Silva and G. Ricarte,
\textit{Geometric gradient estimates for fully nonlinear models with nonhomogeneous degeneracy and applications},
Calc. Var. Partial Differential Equations \textbf{59} (2020), no. 5, Paper 161.

\bibitem{SV2} J. V. Da Silva and H. Vivas,
  \textit{The obstacle problem for a class of degenerate fully nonlinear operators},
  Rev. Mat. Iberoam. \textbf{37} (2021), no. 5, 1991-2020. 
   
   
\bibitem{SV1} J. V. Da Silva and H. Vivas,
\textit{Sharp regularity for degenerate obstacle type problems: a geometric approach},
Discrete Contin. Dyn. Syst. \textbf{41} (2021), no. 3,  1359-1385.

\bibitem{DFQ1} J. D\'avila, P. Felmer, and A. Quaas,
\textit{Alexandroff-Bakelman-Pucci estimate for singular or degenerate fully nonlinear elliptic equations},
C. R. Math. Acad. Sci. Paris \textbf{347} (2009), no. 19-20, 1165-1168. 

\bibitem{DFQ2} J. D\'avila, P. Felmer, and A. Quaas,
\textit{Harnack inequality for singular fully nonlinear operators and some existence results},
Calc. Var. Partial Differential Equations \textbf{39} (2010), no. 3-4, 557–578. 


\bibitem{De1} C. De Filippis,
\textit{Regularity for solutions of fully nonlinear elliptic equations with nonhomogeneous degeneracy},
Proc. Roy. Soc. Edinburgh Sect. A \textbf{151} (2021), no. 1, 110-132.

\bibitem{dSJRR1} J. V. da Silva, E. C. Júnior, G. Rampasso, and G. C. Ricarte,  
\textit{Global regularity for a class of fully nonlinear PDEs with unbalanced variable degeneracy},
arXiv preprint arXiv:2108.08343 (2021).


\bibitem{FRZ1} Y. Fang, V. D. Radulescu, and C. Zhang, 
\textit{Regularity of solutions to degenerate fully nonlinear elliptic equations with variable exponent}, 
Bull. Lond. Math. Soc. \textbf{53} (2021), no. 6, 1863-1878. 

\bibitem{GT1} D. Gilbarg and N. S. Trudinger, 
\textit{Elliptic partial differential equations of second order},
Second edition. Grundlehren der mathematischen Wissenschaften [Fundamental Principles of Mathematical Sciences], 224. Springer-Verlag, Berlin, 1983. xiii+513 pp.



\bibitem{Im1} C. Imbert,
\textit{Alexandroff-Bakelman-Pucci estimate and Harnack inequality for degenerate/singular fully non-linear equations},
J. Differential Equations \textbf{250} (2011), no. 3,  1553-1574.


\bibitem{IS1} C. Imbert and L. Silvestre,
\textit{$C^{1,\alpha}$ regularity of solutions of some degenerate fully non-linear elliptic equations},
Adv. Math. \textbf{233} (2013), 196-206.




\bibitem{Ju1} T. Junges Miotto, 
\textit{The Aleksandrov-Bakelman-Pucci estimates for singular fully nonlinear operators},
Commun. Contemp. Math. \textbf{12} (2010), no. 4, 607–627. 


\bibitem{KK1} B. Kawohl and N. Kutev, 
\textit{Comparison principle and {L}ipschitz regularity for viscosity solutions of some classes of nonlinear partial differential equations},
Funkcial. Ekvac. \textbf{43} (2000), no. 2, 241–253. 
		

\bibitem{KK2} B. Kawohl and N. Kutev, 
\textit{Comparison principle for viscosity solutions of fully nonlinear, degenerate elliptic equations},
  Comm. Partial Differential Equations \textbf{32} (2007), no. 7-9, 1209-1224. 


\bibitem{LZ1} Y. Lian and K. Zhang,
\textit{Boundary pointwise {$C^{1,\alpha}$} and {$C^{2,\alpha}$}
              regularity for fully nonlinear elliptic equations},
   J. Differential Equations \textbf{269} (2020), no. 2, 1172–1191. 
   




\end{thebibliography}

\end{document}